\documentclass{elsarticle}

\usepackage{graphicx}
\usepackage{a4wide}
\usepackage{amstext, amsmath, amssymb}
\usepackage{latexsym}
\usepackage{booktabs}
\usepackage{url}
\usepackage{import}
\usepackage{fancyvrb}
\usepackage{stmaryrd}
\usepackage{pdfsync}
\usepackage{algorithm}
\usepackage{tikz}
\usepackage{pgfplots}
\usepackage{pgfplotstable}
\usepackage{tabularx}
\usepackage{todonotes}
\usepackage{caption}
\usepackage{subcaption}
\usepackage[utf8]{inputenc}					
\usepackage[T1]{fontenc}				
\usepackage[style=english]{csquotes}
\usepackage{enumerate}
\usepackage[english]{babel}

\usepackage[plainpages=false]{hyperref}
\hypersetup{ bookmarksopen=true,
		bookmarksnumbered=true,
		pdfborder={0 0 0},
                pdftitle= {\bf A Full Gradient Stabilized Cut Finite Element Method for
Surface Partial Differential Equations},
                pdfauthor={E. Burman, P. Hansbo, M.G. Larson, A. Massing, S. Zahedi},
                pdfcreator={E. Burman, P. Hansbo, M.G. Larson, A. Massing, S. Zahedi},
		colorlinks=true,
		linkcolor=blue,
		urlcolor=blue	
	}

\biboptions{numbers,square,sort&compress}

\newcommand{\RR}{\mathbb{R}}

\newcommand{\bfzero}{\boldsymbol 0}

\newcommand{\mcK}{\mathcal{K}}

\newcommand{\mcF}{\mathcal{F}}
\newcommand{\mcA}{\mathcal{A}}

\newcommand{\mcT}{\mathcal{T}}
\newcommand{\mcH}{\mathcal{H}}

\newcommand{\mcX}{\mathcal{X}}

\newcommand{\jump}[1]{[#1]}

\newcommand{\Gammah}{{\Gamma_h}}
\newcommand{\nablas}{\nabla_\Gamma}
\newcommand{\nablash}{\nabla_{\Gamma_h}}

\newcommand{\Ps}{{P}_\Gamma}

\newcommand{\Psh}{{P}_{\Gamma_h}}
\newcommand{\Qsh}{{Q}_{\Gamma_h}}

\newcommand{\foralls}{\forall\,}

\newcommand{\dx}{\,\mathrm{d}x}

\newcommand{\IR}{\mathbb{R}}

\DeclareMathOperator{\spann}{span}
\DeclareMathOperator{\dist}{dist}

\usepackage[plainpages=false]{hyperref}
\hypersetup{
		bookmarksopen=true,
		bookmarksnumbered=true,
		pdfborder={0 0 0},
                pdftitle= {Full Gradient Stabilized Cut Finite Element Methods for
Surface Partial Differential Equations},
                pdfauthor={E. Burman, P. Hansbo, M.G. Larson, A. Massing, S. Zahedi},
                pdfcreator={E. Burman, P. Hansbo, M.G. Larson, A. Massing, S. Zahedi},
		colorlinks=true,
		linkcolor=blue,
		urlcolor=blue	
	}

\DefineVerbatimEnvironment{code}{Verbatim}{frame=single,rulecolor=\color{blue}}

\numberwithin{equation}{section}

\newtheorem{lemma}{Lemma}[section]

\newtheorem{theorem}{Theorem}[section]

\newdefinition{remark}[section]{Remark}
\newproof{proof}{Proof}

\journal{\phantom{journal}}

\begin{document}
\begin{frontmatter}

  \title{Full Gradient Stabilized Cut Finite Element Methods for
Surface Partial Differential Equations}

\author[ucl]{Erik Burman}
\ead{e.burman@ucl.ac.uk}

\author[ju]{Peter Hansbo}
\ead{peter.hansbo@ju.se }

\author[umu]{Mats G.\ Larson}
\ead{mats.larson@umu.se}

\author[umu]{Andr\'e Massing\corref{cor1}}
\ead{andre.massing@umu.se}
\cortext[cor1]{Corresponding author}

\author[kth]{Sara Zahedi}
\ead{sara.zahedi@math.kth.se}

\address[ucl]{Department of Mathematics, University College London, London, UK--WC1E 6BT, United Kingdom}
\address[ju]{Department of Mechanical Engineering, J\"onk\"oping University, SE-55111 J\"onk\"oping, Sweden.}
\address[umu]{Department of Mathematics and Mathematical Statistics, Ume{\aa} University, SE-90187 Ume{\aa}, Sweden}
\address[kth]{Department of Mathematics, KTH, SE-10044 Stockholm, Sweden}

\begin{abstract} 
We propose and analyze a new stabilized cut finite element
method for the Laplace-Beltrami operator on a closed
surface. The new stabilization term provides control of
the full $\IR^3$ gradient on the active mesh consisting of
the elements that intersect the surface. Compared to face
stabilization, based on controlling the jumps in the
normal gradient across faces between elements in the
active mesh, the full gradient stabilization is easier to
implement and does not significantly increase the number
of nonzero elements in the mass and stiffness matrices.
The full gradient stabilization term may be combined with
a variational formulation of the Laplace-Beltrami operator
based on tangential or full gradients and we present a
simple and unified analysis that covers both cases.  The
full gradient stabilization term gives rise to a
consistency error which, however, is of optimal order for
piecewise linear elements, and we  obtain optimal order a
priori error estimates in the energy and $L^2$ norms as
well as an optimal bound of the condition number. Finally,
we present detailed numerical examples where we in
particular study the sensitivity of the condition number
and error on the stabilization parameter.
\end{abstract}

\begin{keyword}
  Surface PDE \sep Laplace-Beltrami operator \sep cut finite element method
  \sep stabilization \sep condition number \sep a priori error
  estimates
\end{keyword}

\end{frontmatter}

\section{Introduction}

Cut finite elements have recently been proposed in \cite{OlshanskiiReuskenGrande2009} as a new method for 
the solution of partial differential equations on surfaces 
embedded in $\IR^3$. The main idea is to use the restriction 
of basis functions defined on a three dimensional 
(background) mesh to a discrete surface that is 
allowed to cut through the mesh in an arbitrary 
fashion. The active mesh consists of all 
elements that intersect the discrete surface. This 
approach yields a potentially ill posed stiffness 
matrix and therefore either preconditioning 
\cite{OlshanskiiReusken2014} or 
stabilization \cite{BurmanHansboLarson2015} is 
necessary. The stabilization proposed in 
\cite{BurmanHansboLarson2015} is based on adding 
a consistent stabilization term that provides control 
of the jump in the normal gradient on each of the 
interior faces in the active mesh. Further developments 
in this area include convection problems on surfaces 
\cite{OlshanskiiReuskenXu2014,BurmanHansboLarsonEtAl2015}, adaptive
methods \cite{ChernyshenkoOlshanskii2014,DemlowOlshanskii2012},
coupled surface bulk problems
\cite{BurmanHansboLarsonEtAl2014,GrossOlshanskiiReusken2014}, and time
dependent problems
\cite{OlshanskiiReuskenXu2014a,OlshanskiiReusken2014,HansboLarsonZahedi2015b,HansboLarsonZahedi2016}.
See also the review article \cite{BurmanClausHansboEtAl2014} on cut
finite element methods and references therein, and
\cite{DziukElliott2013} for a general background on finite element
methods for surface partial differential equations.

In this contribution we propose and analyze a new stabilized 
cut finite element method for the Laplace-Beltrami operator 
on a closed surface, which is based on adding a stabilization 
term that provides control of the full $\IR^3$ gradient on the 
active mesh. The advantage of the full gradient 
stabilization compared to face stabilization is that the 
full gradient stabilization term is an elementwise quantity and 
thus is very easy to implement and, more importantly, it does not 
significantly increase the number of nonzero elements in the stiffness 
matrix. 

The full gradient stabilization may be used in combination 
with a variational formulation of the Laplace-Beltrami operator based on 
tangential gradients or full gradients. In the latter case we end up 
with a simple formulation only involving full gradients. Both the 
full gradient stabilization term and variational formulation are based 
on the observation that the extension of the exact solution is constant 
in the normal direction and thus its normal gradient is zero. Since we 
are using the full gradient and not the normal part of the gradient the 
stabilization term gives rise to a consistency error which, 
however, is of optimal order for piecewise linear elements. Using the 
full gradient in the variational formulation was proposed \citet{DeckelnickElliottRanner2013} where, however, no additional 
stabilization term was included. Furthermore, it was shown in 
\cite{Reusken2013} that when the full gradient was used preconditioning 
also works.

Assuming that the discrete surface satisfies standard geometry approximation properties we show 
optimal order a priori error estimates in the energy and $L^2$ 
norms. Furthermore, we show an optimal bound on the condition 
number. Finally, we present numerical examples verifying the 
theoretical results. 
In particular, we study the sensitivity of the accuracy and the
condition number with respect to the choice of the stabilization
parameter for both full gradient and face stabilized methods and
conclude that the sensitivity is in fact considerably smaller for the
full gradient stabilization.

The outline of the paper is as follows: In Section 2 we present the 
model problem, some notation, and the finite element method, in
Section 3 we summarize the necessary preliminaries for our error
estimates, in Section 4 we show stability estimates and the optimal
bound of the condition number, in Section 5 we prove the a priori
error error estimates, and in Section 6 we present some numerical
examples. 

\section{Model Problem and Finite Element Method}

\subsection{The Continuous Surface}
\label{ssec:preliminaries}
In what follows, $\Gamma$ denotes a smooth compact hypersurface
without boundary which is embedded in ${{\RR}}^{d}$ and equipped with a
normal field $n: \Gamma \to \RR^{d}$ and signed distance
function $\rho$.
Defining
the tubular neighborhood of $\Gamma$ by
$U_{\delta_0}(\Gamma) = \{ x \in \RR^{d} : \dist(x,\Gamma) < \delta_0
\}$, the closest point projection $p(x)$ 
is the uniquely defined mapping given by
\begin{align}
  p(x) =  x - \rho(x) n(p(x))
\label{eq:closest-point-projection}
\end{align}
which maps $x \in U_{\delta_0}(\Gamma)$ to the unique point $p(x) \in
\Gamma$ such that $| p(x) - x | = \dist(x, \Gamma)$ for some $\delta_0 > 0$, see~\citet{GilbargTrudinger2001}.
The closest point projection allows the extension of a function $u$ on $\Gamma$ to its
tubular neighborhood $U_{\delta_0}(\Gamma)$ using the pull back
\begin{equation}
\label{eq:extension}
u^e(x) = u \circ p (x)
\end{equation}
In particular, we can smoothly extend the normal field $n_{\Gamma}$ to the tubular neighborhood $U_{\delta_0}(\Gamma)$.
On the other hand, 
for any subset $\widetilde{\Gamma} \subseteq U_{\delta_0}(\Gamma)$ such that 
$p: \widetilde{\Gamma} \to \Gamma $ is bijective, a function $w$ on
$\widetilde{\Gamma}$ can be lifted to $\Gamma$ by the push forward
\begin{align}
  (w^l(x))^e = w^l \circ p = w \quad \text{on } \widetilde{\Gamma}
\end{align}

A function $u: \Gamma \to \RR$ is of class $C^l(\Gamma)$ 
if there exists an extension $\overline{u} \in C^l(U)$ with $\overline{u}|_{\Gamma} = u$
for some $d$-dimensional neighborhood $U$ of $\Gamma$.
Then the tangent gradient $\nabla_\Gamma$ on $\Gamma$ is defined
by
\begin{equation}
\nablas u = \Ps \nabla \overline{u}
\label{eq:tangent-gradient}
\end{equation}
with $\nabla$ the ${{\RR}}^{d}$ gradient and $\Ps = \Ps(x)$ the
orthogonal projection of $\RR^{d}$ onto the tangent plane of $\Gamma$ at $x \in \Gamma$
given by
\begin{equation}
  \Ps = I - n \otimes n
\end{equation}
where $I$ is the identity matrix. It can easily be shown that the definition~\eqref{eq:tangent-gradient} is independent of the extension $\overline{u}$.
We let $\| w \|^2_\Gamma = (w,w)_\Gamma $ denote the $L^2(\Gamma)$ norm on $\Gamma$ and
introduce the Sobolev $H^m(\Gamma)$ space as the subset of
$L^2$ functions for which the norm
\begin{equation}
\| w \|^2_{m,\Gamma} = \sum_{k=0}^m \| D^{P,k}_\Gamma w\|_\Gamma^2,
\quad m = 0,1,2
\end{equation}
is defined. Here, the $L^2$ norm for a matrix is based on the pointwise
Frobenius norm, $D^{P,0}_\Gamma w = w$ and the derivatives $D^{P,1}_\Gamma = \Ps\nabla w,
D^{P,2}_\Gamma w = \Ps(\nabla \otimes \nabla w)\Ps$ are taken in a weak sense.
Finally, for any function space $V$ defined on $\Gamma$, we denote
the space consisting of extended functions by $V^e$ and
correspondingly, we use the notation $V^l$ to refer to the lift of a
function space~$V$ defined on $\widetilde{\Gamma}$.

\subsection{The Continuous Problem}
We consider the following problem: find $u: \Gamma \rightarrow {{\RR}}$
such that
\begin{align}
  \label{eq:LB}
-\Delta_\Gamma u = f \quad \text{on $\Gamma$}
\end{align}
where $\Delta_\Gamma$ is the Laplace-Beltrami operator on $\Gamma$
defined by
\begin{equation}
\Delta_\Gamma = \nabla_\Gamma \cdot \nabla_\Gamma
\end{equation}
and $f\in L^2(\Gamma)$ satisfies $\int_\Gamma f = 0$. The
corresponding weak statement takes the form: find
$u \in H^1(\Gamma)/\RR$
such that
\begin{equation}
a(u,v) = l(v) \quad \forall v \in H^1(\Gamma)/\RR
\end{equation}
where
\begin{equation}
a(u,v) = (\nabla_\Gamma u, \nabla_\Gamma v)_\Gamma, \qquad l(v) = (f,v)_\Gamma
\end{equation}
and $(v,w)_\Gamma = \int_\Gamma v w$ is the $L^2$ inner product.
It follows from the Lax-Milgram lemma that this problem has a unique
solution. For smooth surfaces we also have the elliptic regularity
estimate
\begin{equation}
\| u \|_{2,\Gamma} \lesssim \|f\|_\Gamma
\label{eq:ellreg}
\end{equation}

Here and throughout the paper we employ the notation $\lesssim$ to denote less
or equal up to a positive constant that is always independent of the
mesh size. The binary relations $\gtrsim$ and $\sim$ are defined analogously. 

\subsection{The Discrete Surface and Cut Finite Element Space}
\label{ssec:domain-and-fem-spaces}

Let $\widetilde{\mcT}_{h}$ be a quasi uniform mesh, with mesh parameter
$0<h\leq h_0$, consisting of shape regular simplices of an open and bounded domain $\Omega$ in
$\RR^{d}$ containing $U_{\delta_0}(\Gamma$). 
On $\widetilde{\mcT}_h$, let $\rho_h$ be a continuous, piecewise
linear approximation of the signed distance function $\rho$
and define the discrete surface $\Gamma_h$ as the zero level set of
$\rho_h$,
\begin{equation}
\Gamma_h = \{ x \in \Omega : \rho_h(x) = 0 \}
\end{equation}
We note that $\Gamma_h$ is a polygon with flat faces and we let
$n_h$ be the piecewise constant exterior unit normal to $\Gamma_h$.
We assume that: 
\begin{itemize}
\item $\Gamma_h \subset U_{\delta_0}(\Gamma)$ and that the closest
point mapping $p:\Gamma_h \rightarrow \Gamma$ is a bijection for $0< h
\leq h_0$.
\item The following estimates hold
\begin{equation} \| \rho \|_{L^\infty(\Gamma_h)} \lesssim h^2, \qquad
\| n^e - n_h \|_{L^\infty(\Gamma_h)} \lesssim h
\label{eq:geometric-assumptions-II}
\end{equation} 
\end{itemize}
These properties are, for instance, satisfied if
$\rho_h$ is the Lagrange interpolant of $\rho$.
For the background mesh $\widetilde{\mcT}_{h}$, we define the active (background)
$\mcT_h$
mesh and its set of interior faces $\mcF_h$ by
\begin{align} 
  \mcT_h &= \{ T \in \widetilde{\mcT}_{h} : T \cap \Gamma_h \neq \emptyset \}
  \label{eq:narrow-band-mesh}
  \\
  \mcF_h &= \{ F =  T^+ \cap T^-: T^+, T^- \in \mcT_h \}
  \label{eq:interior-faces}
\end{align}
The active mesh induces a partition $\mcK_h$ of
the approximated surface geometry $\Gamma_h$:
\begin{align}
  \mcK_h&=\{K = \Gamma_h \cap T : T \in \mcT_h \}
\end{align}
The various set of geometric entities are illustrated in
Figure~\ref{fig:domain-set-up}.
\begin{figure}[htb]
  \begin{center}
    \includegraphics[width=0.45\textwidth]{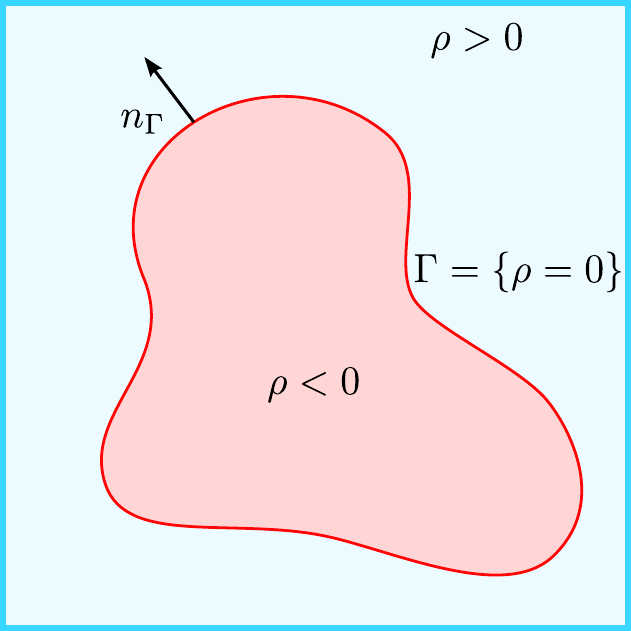}
    \hspace{0.03\textwidth}
    \includegraphics[width=0.45\textwidth]{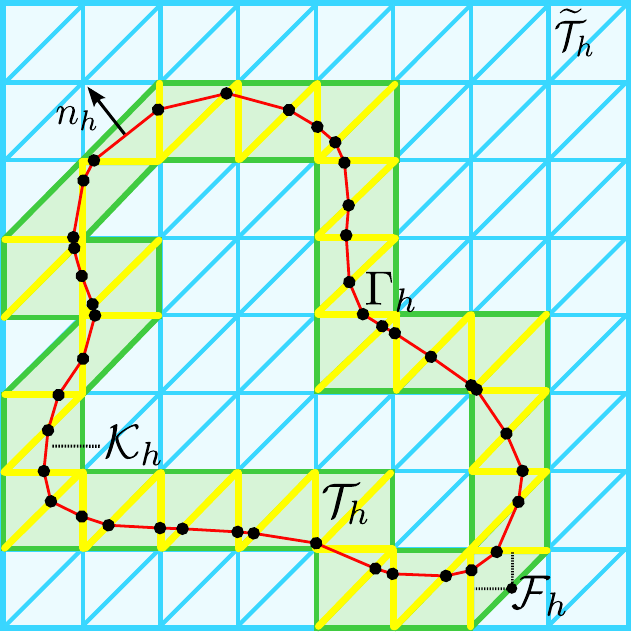}
  \end{center}
  \caption{Domain set-up}
  \label{fig:domain-set-up}
\end{figure}
We observe that the active  mesh $\mcT_h$ gives raise to a discrete or approximate
$h$-tubular neighborhood of $\Gamma_h$, which we denote by
\begin{align}
 \mcT_h = \cup_{T \in \mcT_h} T
\end{align}
Note that for all elements
$T \in \mcT_h$ there is a neighbor $T'\in \mcT_h$ such that $T$ and
$T'$ share a face. 
Finally, let
\begin{equation}
  V_h = \{ v \in C(\mcT_h) : v|_T \in  P_1(T) \}
\label{eq:Vh-def}
\end{equation}
be the space of continuous piecewise linear polynomials defined
on $\mcT_h$ and define the 
discrete counterpart of $H^1(\Gamma)/\RR$
by
\begin{align}
  V_{h,0} = \{ v \in V_h : \lambda_{\Gamma_h}(v) = 0 \}
\label{eq:Vh0-def}
\end{align}
consisting of those $v \in V_h$ with zero average $\lambda_{\Gamma_h}(v) =
\int_{\Gamma_h} v$. 

\subsection{The Full Gradient Stabilized Cut Finite Element Method}
As the discrete counterpart of the bilinear form $a(\cdot,\cdot)$
we consider, similar to~\cite{Reusken2014}, both a tangential and full gradient 
variant,
\begin{align}
  a_h^1(v,w) &= (\nablash v, \nablash w)_{\mcK_h}
  \label{eq:ahgamma-def}
  \\
  a_h^2(v,w) &= (\nabla v, \nabla w)_{\mcK_h}
  \label{eq:ah-def}
\end{align}
Defining the discrete linear form 
\begin{equation}
l_h(v) = ( f^e, v )_{\mcK_h}
\end{equation}
the full gradient stabilized cut finite element method for the
Laplace-Beltrami problem~\eqref{eq:LB}
takes the form: find  $u_h \in V_{h,0}$ such that for $i=1,2$
\begin{equation}
  A_h^i(u_h,v) = l_h(v) \quad \forall v \in V_{h,0}
\label{eq:weak-cutfem-formulation}
\end{equation}
with
\begin{equation}
A_h^i(v,w) = a_h^i(v,w) + \tau s_{h}(v,w) \quad \forall v,w \in V_h
  \label{eq:Ah-def}
\end{equation}
where $\tau$ is a positive parameter
and $s_h(\cdot,\cdot)$ is the full gradient stabilization defined by
\begin{align}
  s_h(v,w) &= h(\nabla v, \nabla w)_{\mcT_h}
  \label{eq:sh-def}
\end{align}
As the forthcoming
a priori error and condition number analysis 
of the first formulation will be covered by the analysis of the second,
we will from now focus on the latter one and omit the superscript $i$.
We introduce the stabilization norm
\begin{align}
  \| v \|_{s_h}^2 = s_h(v,v)
  \label{eq:sh-norm}
\end{align}
as well as the following energy norms
for $v \in H^1(\Gamma) + V_h^l$ and $w \in H^1(\Gamma)^e + V_h$
\begin{align}
  \| v \|_{a}^2 = a(v,v), \quad
  \| w \|_{a_h}^2 = a_h(w,w),  \quad
  \| w \|_{A_h}^2 = A_h(w,w) = \|w\|_{a_h}^2 + \|w\|_{s_h}^2
  \label{eq:energy-norms}
\end{align}
Clearly, the bilinear form~\eqref{eq:Ah-def}
is both coercive and continuous with respect to
$\|\cdot\|_{A_h}$:
\begin{align}
  \|v\|_{A_h}^2 &\lesssim A_h(v,v)
  \qquad 
  \\
  A_h(v,w) &\lesssim \| v \|_{A_h} \| w \|_{A_h}
  \label{eq:Ah-boundedness}
\end{align}

\section{Preliminaries}

\subsection{Trace Estimates and Inverse Inequalities}
First, we recall the following trace inequality for $v \in H^1(\mcT_h)$
\begin{equation}
  \| v \|_{\partial T} 
  \lesssim
  h^{-1/2} \|v \|_{T} +
  h^{1/2} \|\nabla v\|_{T}
  \quad \foralls T \in \mcT_h
  \label{eq:trace-inequality}
\end{equation}
while for the intersection $\Gamma \cap T$ the corresponding inequality
\begin{align}
  \| v \|_{\Gamma \cap T} 
  \lesssim
  h^{-1/2} \| v \|_{T} 
  + h^{1/2} \|\nabla v\|_{T} 
  \quad \foralls T \in \mcT_h
  \label{eq:trace-inequality-cut-faces}
\end{align}
holds whenever $h$ is small enough,
see \cite{HansboHansboLarson2003} for a proof.
In the following, we will also need some
well-known inverse estimates for $v_h \in V_h$:
\begin{gather}
  \label{eq:inverse-estimate-grad}
  \| \nabla v_h\|_{T} 
  \lesssim
  h^{-1} 
  \| v_h\|_{T} \quad \foralls T \in \mcT_h
  \\
  \| v_h\|_{\partial T} 
  \lesssim
  h^{-1/2} 
  \| v_h\|_{T},
  \qquad 
  \| \nabla v_h\|_{\partial T} 
  \lesssim
  h^{-1/2} 
  \| \nabla v_h\|_{T} \quad \foralls T \in \mcT_h
  \label{eq:inverse-estimates-boundary}
\end{gather}
and the following ``cut versions'' 
when $K \cap T \not \subseteq \partial T$
\begin{alignat}{5}
  \|v_h \|_{K \cap T} 
  &\lesssim
  h^{-1/2} \|v_h\|_{T},
  & & \qquad 
  \| \nabla v_h \|_{K \cap T} 
  &\lesssim
  h^{-1/2} \|\nabla v_h\|_{T}
  & &\quad \foralls K \in \mcK_h, \;
  \foralls T \in \mcT_h
  \label{eq:inverse-estimate-cut-v-on-K}
\end{alignat}
which are an immediate consequence of similar inverse estimates
presented in~\cite{HansboHansboLarson2003}.

\subsection{Geometric Estimates}
We now recall some standard geometric identities and estimates which
typically are used in numerical analysis of the discrete scheme when
passing from the discrete surface to the continuous one and vice versa.
For a detailed derivation, we refer to 
\citet{Dziuk1988,OlshanskiiReuskenGrande2009,DziukElliott2013}.
Starting with the Hessian of the signed distance function
\begin{align}
  \mcH = \nabla \otimes \nabla \rho \quad \text{on }
  U_{\delta_0}(\Gamma)
\end{align}
the derivative of the closest point projection 
and of an extended function $v^e$ is given by
\begin{gather}
Dp = \Ps (I - \rho \mcH) = \Ps - \rho \mcH
\label{eq:derivative-closest-point-projection}
\\
  Dv^e = D(v \circ p) = Dv Dp = Dv P_{\Gamma}(I - \rho \mcH)
\label{eq:derivative-extended-function}
\end{gather}
The self-adjointness of $\Ps$, $\Psh$, and $\mcH$,
and the fact that $ \Ps \mcH = \mcH = \mcH \Ps$
and $\Ps^2 = \Ps$
then leads to the identities
\begin{align}
  \nabla v^e &= \Ps(I - \rho \mcH) \nabla v
  = \Ps(I - \rho \mcH) \nablas v
  \label{eq:ve-full gradient}
  \\
  \nablash v^e &= \Psh(I - \rho \mcH)\Ps \nabla v = B^{T} \nablas v
  \label{eq:ve-tangential-gradient}
\end{align}
where the invertible linear mapping
\begin{align}
  B = P_{\Gamma}(I - \rho H) P_{\Gamma_h}: T_x(\Gammah) \to T_{p(x)}(\Gamma)
  \label{eq:B-def}
\end{align}
maps the tangential space of $\Gamma_h$ at $x$ to the tangential space of $\Gamma$ at
$p(x)$. Setting $v = w^l$ and using the identity $(w^l)^e = w$, we immediately get that
\begin{align}
  \nablas w^l = B^{-T} \nablash w
\end{align}
for any elementwise differentiable function $w$ on $\Gamma_h$ lifted to $\Gamma$.
We recall from \cite[Lemma 14.7]{GilbargTrudinger2001}
that for $x\in U_{\delta_0}(\Gamma)$, the Hessian $\mcH$
admits a representation
\begin{equation}\label{Hform}
  \mcH(x) = \sum_{i=1}^d \frac{\kappa_i^e}{1 + \rho(x)\kappa_i^e}a_i^e \otimes a_i^e
\end{equation}
where $\kappa_i$ are the principal curvatures with corresponding
principal curvature vectors $a_i$.
Thus
\begin{equation}
  \|\mcH\|_{L^\infty(U_{\delta_0}(\Gamma))} \lesssim 1
  \label{eq:Hesse-bound}
\end{equation}
for $\delta_0 > 0$ small enough and as a consequence
the following bounds for the linear
operator $B$ can be derived (see \cite{Dziuk1988,DziukElliott2013} for
the details):
\begin{lemma} It holds
  \label{lem:BBTbound}
  \begin{equation}
    \| B \|_{L^\infty(\Gamma_h)} \lesssim 1,
    \quad \| B^{-1} \|_{L^\infty(\Gamma)} \lesssim 1,
    \quad
    \| P_\Gamma - B B^T \|_{L^\infty(\Gamma)} \lesssim h^2
    \label{eq:BBTbound}
  \end{equation}
\end{lemma}

In the course of the a priori analysis in Section~\ref{sec:a-priori-est}, 
we will need to quantify the error introduced by using the
full gradient in~\eqref{eq:ah-def} instead of $\nablash$.
To do so we decompose the full gradient as
$\nabla = \nablash + \Qsh \nabla$ 
with $\Qsh = I - \Psh = n_h \otimes n_h$.
We then have
\begin{lemma}
  For $v \in H^1(\Gamma)$ it holds
  \begin{align}
    \| \Qsh \nabla v^e \|_{\Gamma}
    \lesssim h 
    \| \nablas v \|_{\Gamma} 
    \label{eq:normal-grad-est}
  \end{align}
\end{lemma}
\begin{proof}
  Since $\nabla v^e$ = $ \Ps(I - \rho \mcH)\nablas v$
  according to identity~\eqref{eq:ve-full gradient}, it is enough to
  prove that
  \begin{align}
    \| \Qsh \Ps \|_{L^{\infty}(\Gamma)} \lesssim h
    \label{eq:normal-tangential-est}
  \end{align}
  But a simple computation shows that
\begin{align}
  \| \Qsh \Ps \|_{L^{\infty}(\Gamma)}
  &= 
  \| n_h \otimes n_h - (n_h,n)_{\RR^d} n_h \otimes n \|_{L^{\infty}(\Gamma)}
  \\
  &= 
  \|(1 - (n_h,n)_{\RR^d}) n_h \otimes n_h \|_{L^{\infty}(\Gamma)} + \| (n_h,n)_{\RR^d} n_h
  \otimes (n_h - n) \|_{L^{\infty}(\Gamma)}
  \\
  &\lesssim h^2 + h
\end{align}
\end{proof}
Next, 
for a subset $\omega\subset \Gammah$,
we have the change of variables formula
\begin{equation}
\int_{\omega^l} g^l d\Gamma 
=  \int_{\omega} g |B|d\Gamma_h
\end{equation}
with $|B|$ denoting the absolute value of the determinant 
of $B$.  The determinant $|B|$ satisfies the following estimates.
\begin{lemma}  It holds
  \label{lem:detBbounds}
  \begin{alignat}{5}
  \| 1 - |B| \|_{L^\infty(\mcK_h)} 
  &\lesssim h^2, 
  & &\qquad
  \||B|\|_{L^\infty(\mcK_h)} 
  &\lesssim 1, 
  & &\qquad
  \||B|^{-1}\|_{L^\infty(\mcK_h)} 
  &\lesssim 1
    \label{eq:detBbound}
\end{alignat}
\end{lemma}
\noindent Combining the various estimates for the norm and the determinant of $B$ shows
that for $m = 0,1$
\begin{alignat}{3}
  \| v \|_{H^{m}(\mcK_h^l)} &\sim \| v^e \|_{H^{m}(\mcK_h)}  
  & &\quad \text{for } v \in H^m(\mcK_h^l)
  \label{eq:norm-equivalences-ve}
  \\
  \| w^l \|_{H^{m}(\mcK_h^l)} &\sim \| w \|_{H^{m}(\mcK_h)}
  & &\quad \text{for } w \in V_h
  \label{eq:norm-equivalences-wh}
\end{alignat}

Next, we observe that thanks to the coarea-formula (cf. \citet{EvansGariepy1992})
\begin{align*}
\int_{U_{\delta}} f(x) \,dx = \int_{-\delta}^{\delta} 
\left(\int_{\Gamma(r)} f(y,r) \, \mathrm{d} \Gamma_r(y)\right)\,\mathrm{d}r
\end{align*}
the extension operator $v^e$ defines a bounded operator
$H^m(\Gamma) \ni v \mapsto v^e \in H^m(U_{\delta}(\Gamma))$
satisfying the stability estimate
\begin{align}
  \| v^e \|_{k,U_{\delta}(\Gamma)} \lesssim \delta^{1/2} \| v
  \|_{k,\Gamma}, \qquad 0 \leqslant k \leqslant m
\label{eq:stability-estimate-for-extension}
\end{align}
for $0 < \delta \leqslant \delta_0$ where the hidden constant depends only on the curvature of $\Gamma$.

\subsection{Interpolation Operator}
Next, we recall from~\citet{ScottZhang1990} that for
$v \in H^m(\mcT_h)$,
the standard Scott-Zhang interpolant $\pi_h:L^2(\mcT_h) \rightarrow V_h$ satisfies the
local interpolation estimates
\begin{alignat}{3}
\| v - \pi_h v \|_{k,T} 
& \lesssim
  h^{l-k}| v |_{l,\omega(T)},
  & &\quad 0\leqslant k \leqslant l \leqslant \min\{2,m\} \quad &\foralls T\in \mcT_h
  \label{eq:interpest0}
  \\
\| v - \pi_h v \|_{l,F} &\lesssim h^{l-k-1/2}| v |_{l,\omega(F)},
  & &\quad 0\leqslant k \leqslant l - 1/2 \leqslant \min\{2,m\} - 1/2  \quad &\foralls F\in
  \mcF_{h} 
  \label{eq:interpest1}
\end{alignat}
where $\omega(T)$ consists of all elements sharing a
vertex with $T$. The patch $\omega(F)$ is defined analogously.
Now with the help of the extension operator,
an interpolation operator $\pi_h: H^m(\Gamma) \to V_h$ 
can be constructed by
setting $\pi_h v = \pi_h v^e$, where we took the liberty of using the same symbol.
Choosing $\delta_0 \sim h$,
it follows directly from combining the trace inequality~\eqref{eq:trace-inequality-cut-faces}, the interpolation
estimate~\eqref{eq:interpest1}, and the stability
estimate~\eqref{eq:stability-estimate-for-extension} that
the interpolation operator satisfies the following error estimate:
\begin{lemma}
\label{lem:interpolenergy}
For $v \in H^2(\Gamma)$, it holds that
\begin{align}
\label{eq:interpolenergy}
h \| v^e - \pi_h v^e \|_{\Gammah} 
+ \| v^e - \pi_h v^e \|_{a_h} &\lesssim  h \| v \|_{2,\Gamma}
\end{align}
\end{lemma}

\subsection{Fat Intersection Covering}
\label{ssec:fat-intersection-covering}
Since the surface geometry is represented on fixed background mesh, the active  mesh $\mcT_h$
might contain elements which barely intersects the discretized surface $\Gamma_h$.
Such ``small cut elements'' typically prohibit the application of a whole set of
well-known estimates, such as interpolation estimates and inverse inequalities,
which rely on certain scaling properties.
As a partial replacement for the lost scaling properties we here recall from~\citet{BurmanHansboLarson2015} 
the concept of \emph{fat intersection coverings} of $\mcT_h$.

In \citet{BurmanHansboLarson2015} it was proved that
the active  mesh fulfills a fat intersection property which
roughly states that for every element in $\mcT_h$ there is a 
close-by element which has a significant intersection with $\Gamma_h$.
More precisely, let $x$ be a point on $\Gamma$ and let 
  $B_{\delta}(x) = \{y\in \RR^d: |x-y| < \delta\}$ 
  and 
  $D_{\delta} = B_{\delta}(x) \cap \Gamma$. 
  We define the sets of elements
  \begin{align}
    \mcK_{\delta,x} 
    = \{ K \in \mcK_h : \overline{K}^l \cap D_{\delta}(x) \neq
    \emptyset \},
    \qquad
    \mcT_{\delta,x} 
    = \{ T \in \mcT_h : T \cap \Gamma_h \in \mcK_{\delta,x} \}
    \label{eq:fat-intersection-covering}
  \end{align}
 With $\delta \sim h$ we use the notation $\mcK_{h,x}$ and
$\mcT_{h,x}$. For each $\mcT_h$, $h \in (0,h_0]$ there is a set of
points $\mcX_h$ on $\Gamma$ such that $\{\mcK_{h,x}, x \in \mcX_h \}$
and $\{\mcT_{h,x}, x \in \mcX_h \}$ are coverings of $\mcT_h$ and
$\mcK_h$ with the following properties:
\begin{itemize}
  \item The number of set containing a given point $y$ is uniformly
    bounded
    \begin{align}
      \# \{ x \in \mcX_h : y \in \mcT_{h,x}  \} \lesssim 1 \quad
      \foralls y \in \RR^d
    \end{align}
  for all $h \in (0,h_0]$ with $h_0$ small enough.
\item
  The number of elements in the sets $\mcT_{h,x}$ is uniformly bounded
  \begin{align}
    \# \mcT_{h,x} \lesssim 1 
    \quad \foralls x \in \mcX_h
  \end{align}
  for all $h \in (0,h_0]$ with $h_0$ small enough, and each element in
  $\mcT_{h,x}$ shares at least one face with another element in
  $\mcT_{h,x}$.
\item $\foralls h \in (0,h_0]$ with $h_0$ small enough, and $\foralls x \in \mcX_h$, $\exists
  T_x \in \mcT_{h,x}$ that has a large (fat) intersection with
  $\Gamma_h$ in the sense that
  \begin{align}
  | T_x | \sim h | T_x \cap \Gamma_h | = h | K_x | 
  \quad \foralls x \in \mcX_h
  \end{align}
\end{itemize}
To make use of the fat intersection property in the next section,
we will need the following Lemma~\ref{lem:l2norm-control-via-jumps}
which describes how the control of discrete 
functions on potentially small cut elements can be transferred to
their close-by neighbors with large intersection by using a face-based stabilization term.
A proof of the first estimate can be found in~\citet{MassingLarsonLoggEtAl2013}.
\begin{lemma}
  \label{lem:l2norm-control-via-jumps}
  Let $v \in V_h$ 
  and consider a macro-element $\mathcal{M} = T_1 \cup T_2$ formed by
  any two elements $T_1$ and $T_2$ of $\mcT_h$ sharing a face $F$. 
  Then
  \begin{align}
    \| v \|_{T_1}^2 &\lesssim  \|v\|_{T_2}^2
    + h^3 \| n_F \cdot \jump{\nabla v} \|_F^2 
    \label{eq:l2norm-control-via-jumps-I}
  \end{align}
  with the hidden constant only depending on the quasi-uniformness parameter.
\end{lemma}

\section{Stability and Condition Number estimates}

\subsection{Discrete Poincar\'e Estimates}
\label{ssec:poincare-estimates}

First we recall that $v \in V_h$ satisfies a Poincar\'e inequality
on the surface (see \cite[Lemma 4.1]{BurmanHansboLarson2015}):
\begin{lemma}
  \label{lem:poincare-I} For $v \in V_h$, the following estimate holds 
  \begin{equation}
  \| v - \lambda_{\Gamma_h}(v) \|_{\Gamma_h}
  \lesssim
    \| \nablash v \|_{\Gamma_h} 
  \label{eq:poincare-I}
\end{equation}
  for $0<h \leq h_0$ with $h_0$ small enough.
\end{lemma}
Next, we derive an additional Poincar\'e inequality
which involves a scaled version of the $L^2$ norm
of discrete finite element functions
on the active  mesh.
\begin{lemma}
  \label{lem:discrete-poincare-Nh} For $v \in V_h$, the following estimate holds 
  \begin{equation}
    h^{-1}\| v - \lambda_{\Gamma_h}(v) \|^2_{\mcT_h}
    \lesssim
    \| \nablash v \|_{\Gamma_h}^2 
    + h \| \nabla v \|_{\mcT_h}^2 
  \label{eq:discrete-poincare-Nh}
\end{equation}
  for $0<h \leq h_0$ with $h_0$ small enough.
\end{lemma}
\begin{proof}
  Without loss of generality we can assume that $\lambda_{\Gamma_h}(v) = 0$.
  Apply~\eqref{eq:l2norm-control-via-jumps-I} and
  \eqref{eq:inverse-estimates-boundary} to obtain
  \begin{align}
    \| v \|_{\mcT_h}^2 
    &\lesssim \sum_{x \in \mcX_h} \| v \|_{\mcT_{h,x}}^2
    \lesssim
    \sum_{x \in \mcX_h} \| v \|_{T_x}^2 
    + h^3 \| n_F \cdot \jump{\nabla v} \|_{\mcF_h}^2
    \lesssim
    \sum_{x \in \mcX_h} \| v \|_{T_x}^2 
    + h^2 \| \nabla v \|_{\mcT_h}^2
    \label{eq:poincare-proof-I}
\end{align}
Thus it is sufficient to estimate the first term
in~\eqref{eq:poincare-proof-I}.
For $v \in V_h$, we define a piecewise constant version satisfying
$\overline{v}|_T = \tfrac{1}{|T|} \int_T v \dx$. Clearly 
$\|v - \overline{v}\|_T \lesssim h \| \nabla v \|_T$.
Adding and subtracting $\overline{v}$ gives
\begin{align}
  \sum_{x \in \mcX_h} \| v \|_{T_x}^2
  &\lesssim
  \sum_{x \in \mcX_h}  \| v - \overline{v} \|_{T_x}^2
  +
  \sum_{x \in \mcX_h}  \| \overline{v} \|_{T_x}^2
  \\
  &\lesssim
  h^2 \| \nabla v \|_{\mcT_h}^2
  + \sum_{x \in \mcX_h} h \| \overline{v} \|_{K_x}^2
  \\
  &\lesssim
  h^2 \| \nabla v \|_{\mcT_h}^2
  + h \| v \|_{\Gamma_h}^2
  + 
  \underbrace{h \| v - \overline{v} \|_{\Gamma_h}^2}_{\lesssim h^2 \|
  \nabla v\|_{\mcT_h}^2}
  \\
  &\lesssim
  h^2 \| \nabla v \|_{\mcT_h}^2
  + h \| \nablash v \|_{\Gamma_h}^2 
  \label{eq:transition-to-surface-I}
\end{align}
where in the last step, the Poincar\'e
inequality~\eqref{eq:poincare-I} was applied.
\end{proof}
\begin{remark}
  In \citet{BurmanHansboLarson2015},
  the discrete bilinear form 
  $a^1(v,w) = (\nablash v, \nablash)_{\mcK_h}$ 
  was augmented with the face-based stabilization term 
  \begin{align}
    \tau j_h(v,w) = 
    \tau (n_F \cdot \jump{\nabla v},
    n_F \cdot \jump{\nabla w})_{\mcF_h}
    \label{eq:jh-def}
  \end{align}
  to prove optimal a priori error and condition number estimates using
  a discrete Poincar\'e inequality of the form
  \begin{equation}
    h^{-1}\| v - \lambda_{\Gamma_h}(v) \|^2_{\mcT_h}
    \lesssim
    \| \nablash v \|_{\Gamma_h}^2 
    +  \| n_F \cdot \jump{\nabla v} \|_{\mcF_h}^2 
\end{equation}
for $v \in V_h$. As before, $\tau$ denotes a positive stabilization parameter
which has to be chosen large enough.

Compared to the face-based
stabilization~\eqref{eq:jh-def},
the full gradient stabilization~\eqref{eq:sh-def} 
has three main advantages: Firstly, its implementation
is extremely cheap and immediately available in many 
finite element codes. Secondly, 
the stencil of the discretization operator
is not enlarged, as opposed to using a face-based penalty operator.
Thirdly, anticipating the numerical results in
Section~\ref{sec:numerical_results},
the accuracy and conditioning of a
full gradient stabilized surface method
is less sensitive to the choice of the stability parameter
$\tau$ than for a face-based stabilized scheme.
\label{rem:face-based-vs-full-grad-stab}
\end{remark}

\subsection{Bounds for the Condition Number}
\label{ssec:condition-number-estimate}
With the help of the Poincar\'e estimates derived in the
previous section, we now show that the condition number of the stiffness matrix
associated with the bilinear form~\eqref{eq:weak-cutfem-formulation}
can be bounded by $O(h^{-2})$ independently of the position of the
surface $\Gamma$ relative to the background mesh~$\mcT_h$.
Let $\{\phi_i\}_{i=1}^N$ be the standard piecewise linear basis
functions associated with $\mcT_h$ and thus
$v = \sum_{i=1}^N V_i \phi_i$ for $v \in V_h$ and expansion
coefficients $V = \{V_i\}_{i=1}^N \in \RR^N$.
The stiffness matrix $\mcA$ is given by the relation
\begin{align}
  ( \mcA V, W )_{\RR^N}  = A_h(v, w) \quad \foralls v, w \in
  V_h
  \label{eq:stiffness-matrix}
\end{align}
Recalling the definition of $V_{h,0}$
the stiffness matrix $\mcA$ 
clearly is a bijective linear mapping 
$\mcA:\widehat{\RR}^N \to \ker(\mcA)^{\perp}$
where we set $\widehat{\RR}^N = \RR^N /\ker(\mcA)$
to factor out the one-dimensional kernel given by 
$\ker{\mcA} = \spann\{(1,\ldots,1)^{\top}\}$.
The operator norm and condition number of the matrix $\mcA$ are then defined by
\begin{align}
  \| \mcA \|_{\RR^N}
  = \sup_{V \in \widehat{\RR}^N\setminus\bfzero}
  \dfrac{\| \mcA V \|_{\RR^N}}{\|V\|_{\RR^N}}
\quad \text{and}
\quad
  \kappa(\mcA) = \| \mcA \|_{\RR^N} \| \mcA^{-1} \|_{\RR^N}
  \label{eq:operator-norm-and-condition-number-def}
\end{align}
respectively.
Following the approach in~\citet{ErnGuermond2006}, 
a bound for the condition number can be derived 
by combining the well-known estimate
\begin{align}
  h^{d/2} \| V \|_{\RR^N}
  \lesssim \| v_h \|_{L^2(\mcT_h)}
  \lesssim
  h^{d/2} \| V \|_{\RR^N}
  \label{eq:mass-matrix-scaling}
\end{align}
which holds for any quasi-uniform mesh $\mcT_h$,
with the Poincar\'e-type estimate~\eqref{eq:discrete-poincare-Nh} and
the following inverse estimate:
\begin{lemma} Let $v \in V_{h,0}$ then the following inverse estimate holds
  \label{lem:inverse-estimate-Ah}
  \begin{align}
    \| v \|_{A_h} \lesssim h^{-3/2} \| v \|_{\mcT_h}
    \label{eq:inverse-estimate-Ah}
  \end{align}
  \begin{proof}
   Recall that  $\| v \|_{A_h}=\| v \|_{a_h}+ \| v \|_{s_h}$. First, 
    employ the standard inverse estimate~\eqref{eq:inverse-estimate-grad}
    to obtain
    \begin{align}
     \| v \|_{s_h}^2= h \| \nabla v \|_{\mcT_h}^2
      \lesssim h^{-1} \| v \|_{\mcT_h}^2
    \end{align}
    Next,  the inverse
    estimates~\eqref{eq:inverse-estimate-cut-v-on-K} and
    \eqref{eq:inverse-estimate-grad} gives
    \begin{align}
       \| v \|_{a_h}^2=
      \| \nabla v \|_{\mcK_h}^2
      \lesssim
      h^{-1}\| \nabla v \|_{\mcT_h}^2
      \lesssim
      h^{-3}\| v \|_{\mcT_h}^2
    \end{align}
  which concludes the proof.
  \end{proof}
\end{lemma}
We are now in the position to prove the main result of this section:
\begin{theorem} 
  \label{thm:condition-number-estimate}
  The condition number of the stiffness matrix satisfies
  the estimate
\begin{equation}
\kappa( \mcA )\lesssim h^{-2}
\end{equation}
where the hidden constant depends only on the quasi-uniformness
parameters.
\end{theorem}
\begin{proof} We need to bound $\| \mcA \|_{\RR^N}$ and $\| \mcA^{-1} \|_{\RR^N}$. 
 First observe that for $w \in V_h$,
\begin{equation}
  \| w \|_{A_h}
\lesssim h^{-3/2} \| w \|_{\mcT_h}
\lesssim h^{(d-3)/2}\|W\|_{\RR^N}
\end{equation}
where the inverse estimate~\eqref{eq:inverse-estimate-Ah}
and equivalence~\eqref{eq:mass-matrix-scaling}
were successively used.
Thus
\begin{align}
  \| \mcA V\|_{\RR^N} &= \sup_{W \in \RR^N } 
  \frac{( \mcA V, W)_{\RR^N}}{\| W \|_{\RR^d}}
  = \sup_{w \in V_h }  \frac{A_h(v,w)}{\| w \|_{A_h}} 
  \frac{\| w \|_{A_h}}{\| W \|_{\RR^N}}
\lesssim h^{(d-3)/2} \| v \|_{A_h} 
\lesssim h^{d-3}\|V\|_{\RR^N}
\end{align}
and thus by the definition of the operator norm,
$
\| \mcA \|_{\RR^N} \lesssim h^{d-3}
$.
To estimate $\| \mcA^{-1}\|_{\RR^N}$,
start from \eqref{eq:mass-matrix-scaling} and combine the Poincar\'e
inequality~\eqref{eq:discrete-poincare-Nh} with 
a Cauchy-Schwarz inequality to arrive at the following chain of
estimates:
\begin{align}
  \| V \|^2_{\RR^N} 
  \lesssim h^{-d} \| v \|^2_{\mcT_h} 
  \lesssim h^{1-d} A_h(v,v) 
  = h^{1-d} (V, \mcA V)_{\RR^N}
  \lesssim h^{1-d} \| V \|_{\RR^N} \| \mcA V \|_{\RR^N}
\end{align}
and hence $\| V \|_{\RR^N} \lesssim h^{1-d}\| \mcA V\|_{\RR^N}$. 
Now setting $ V = \mcA^{-1} W$ we conclude that
 $
 \| \mcA^{-1}\|_{\RR^N} \lesssim h^{1-d}
 $
and combining estimates for $\| \mcA\|_{\RR^N}$ and $\| \mcA^{-1}\|_{\RR^N}$ the theorem follows.
\end{proof}

\section{A Priori Error Estimates}
\label{sec:a-priori-est}
This section is devoted to the proof of the main a priori
estimates for the weak formulation~\eqref{eq:weak-cutfem-formulation}.
We proceed in two steps.
First, we establish an abstract Strang-type lemma which reveals that
the overall error can be split into an
interpolation error and a consistency error.
Next, we provide a bound for the consistency error in order to
complete the a priori estimate of the energy norm error.
Finally, using a duality argument, we establish an optimal $L^2$ error 
bound where we use the smoothness of the dual
function to obtain sufficient control of the consistency error.

\subsection{Strang's Lemma}
\begin{lemma} 
  \label{lem:strang}
With $u$ the solution of \eqref{eq:LB}
and $u_h$ the solution of \eqref{eq:weak-cutfem-formulation} it holds
\begin{align}
  \| u^e - u_h \|_{A_h} 
  &\leqslant
  2\| u^e - \pi_h u^e \|_{A_h}
+ \sup_{v \in V_h} 
\dfrac{ 
l_h(v) - A_h(u^e,v)
}{
  \| v \|_{A_h}
}
\end{align}
\end{lemma}
\begin{proof}
Thanks to triangle inequality 
$\| u^e - u_h \|_{A_h} \leqslant \| u^e - \pi_h u^e \|_{A_h}
+ \| u_h - \pi_h u^e \|_{A_h}$, it is sufficient to consider
the discrete error $e_h = u_h - \pi_h u^e $. Now observe that
\begin{align}
  \| e_h \|_{A_h}^2 
  & = A_h(u_h - \pi_h u^e,e_h)
\\
& = 
l_h(e_h) - A_h(u^e, e_h) + A_h(u^e - \pi_h u^e,e_h)
\\
& \lesssim
\sup_{v \in V_h} 
\dfrac{ 
l_h(v) - A_h(u^e,v)
}{
  \| v \|_{A_h}
}
+
\sup_{v \in V_h} 
\dfrac{ 
    A_h(u^e - \pi_h u^e,v)
}{
  \| v \|_{A_h}
}
\label{eq:strang-lemma-step-3}
\end{align}
and apply a Cauchy-Schwarz inequality 
on the second term in~\eqref{eq:strang-lemma-step-3}
to conclude the proof.
\end{proof}

\subsection{Consistency Error Estimates}
Next, we derive a representation of the consistency error, showing
that it can be attributed to a geometric error and a consistency error
introduced by the stabilization form $s_h$.
\label{ssec:consistency-error}
\begin{lemma}
  \label{lem:consistency-error-est}
  Let $v \in V_h$ and $\phi \in H^2(\Gamma)$, then
  the following estimates hold
  \begin{align}
      | l_h(v) - A_h(u^e,v) |
      &\lesssim 
    h \| f\|_{\Gamma} \| v \|_{A_h}
    \label{eq:consistency-error-est-primal}
    \\
      |l_h(\pi_h \phi^e) - A_h(u^e, \pi_h \phi^e)|
      &\lesssim 
    h^2 
    \| f \|_{\Gamma} 
    \| \phi \|_{2,\Gamma} 
    \label{eq:consistency-error-est-dual}
  \end{align}
\end{lemma}
\begin{proof}
  Recalling the definition~\eqref{eq:Ah-def} of $A_h(\cdot,\cdot)$ and inserting
  $a(u,v^l) - l(v^l) = 0$ yields
  \begin{align}
    l_h(v) - A_h(u^e, v)
    &=
    \bigl(
    l_h(v) - 
    l(v^l)
    \bigr)
    + 
    \bigl(a(u,v^l) - a_h(u^e, v)
    \bigr)- s_h(u^e, v)
    \\
    &= I + II + III
  \end{align}
  which we estimate next.

  \noindent  {\bf Term $\boldsymbol{I}$.} 
  For the quadrature error of the right hand side we have
  \begin{align}
    l(v^l)- l_h(v) &= (f,v^l)_\Gamma - (f^e, v)_\Gammah
    = (f,v^l (1 -|B|^{-1}))_\Gamma
    \lesssim h^2  \| f \|_{\Gamma}\|v^l\|_{\Gamma}
    \lesssim h^2  \| f \|_{\Gamma}\| v \|_{a_h}
  \end{align}
  where in the last step, the Poincar\'e
  inequality~\eqref{eq:poincare-I} was used
  after passing from $\Gamma$ to $\Gamma_h$.

  \noindent  {\bf Term $\boldsymbol{II}$.} 
  Using the splitting
  $ \nabla = \nablash + \Qsh \nabla$ 
  the discrete form $a_h$ can be decomposed as
  \begin{align}
    a_h(u^e, v)
    = (\nablash u^e, \nablash v)_{\Gamma_h}
    + (\Qsh \nabla u^e, \Qsh \nabla v)_{\Gamma_h}
  \end{align}
  Inserting this identity into $II$ gives
  \begin{align}
    II = 
    \bigl(
    (\nablas u, \nablas v^l)_{\Gamma}
    - (\nablash u^e, \nablash v)_{\Gamma_h}
    \bigr)
    - (\Qsh \nabla u^e, \Qsh \nabla v)_{\Gamma_h}
    = II_a  + II_b
    \label{eq:ah-split}
  \end{align}
  A bound for the first term $II_a$ can be derived by
  lifting the tangential part of $a_h(\cdot,\cdot)$ to $\Gamma$
  and 
  employing the bound for
  determinant~\eqref{eq:detBbound}
  the operator norm estimates~\eqref{eq:BBTbound},
  and the norm
  equivalences~\eqref{eq:norm-equivalences-ve}--\eqref{eq:norm-equivalences-wh},
  \begin{align}
    II_a 
    &= 
    (\nablas u,\nablas v^l)_{\mcK_h^l} - (\nablash u,\nablash
    v)_{\mcK_h}
    \\
    &=
    (\nablas u,\nablas v^l)_{\mcK_h^l} - ((\nablash u)^l,(\nablash
    v)^l |B|^{-1})_{\mcK_h^l}
    \\
    &=
    ((\Ps - |B|^{-1} B B^T)\nablas u,\nablas v^l)_{\mcK_h^l}
    \\
    &=
    ((\Ps - B B^T)+ (1-|B|^{-1}) B B^T)\nablas u,\nablas v^l)_{\mcK_h^l}
    \\
    &\lesssim h^2 \| f \|_{\Gamma} \|\nablas v^l\|_{\mcK_h^l}
  \end{align}
  Turning to the second term $II_b$ and applying 
  the inequality~\eqref{eq:normal-grad-est} to $\Qsh \nabla u^e$
  gives
  \begin{align}
    II_b 
    & \lesssim 
    \|\Qsh \nabla u^e \|_{\Gammah} 
    \|\Qsh \nabla v \|_{\Gammah} 
    \\
    &
    \lesssim
    h \| f \|_{\Gamma} 
    \|\Qsh \nabla v \|_{\Gammah} 
    \label{eq:estimate-IIb}
  \end{align}
  For general $v \in V_{h}$, the last factor in $II_b$
  is simply bounded by $\| \nabla v \|_{\Gamma_h}$
  while in the special case 
  $v = \pi_h \phi^e$,
  the interpolation estimate~\eqref{eq:interpolenergy}
  and
  a second application of~\eqref{eq:normal-grad-est} to $\Qsh \nabla
  \phi^e$ yields
  \begin{align}
    \|\Qsh \nabla \pi_h \phi^e \|_{\Gammah} 
    \lesssim
    \|\Qsh \nabla \phi^e \|_{\Gammah} 
    +
    \|\Qsh \nabla (\phi^e - \pi_h \phi^e)\|_{\Gammah} 
    \lesssim
    h \| \phi \|_{2,\Gamma}
  \end{align}

\noindent{\bf Term $\boldsymbol{III}$.}
  Combine Cauchy-Schwarz's inequality with the stability
  estimate~\eqref{eq:stability-estimate-for-extension},
  choosing $\delta \sim h$,
  to obtain
  \begin{align}
    s_h(u^e, v) 
    &= h(\nabla u^e, \nabla v)_{\mcT_h}
    \lesssim h^{1/2} \| \nabla u^e \|_{\mcT_h} 
    \| v \|_{s_h}
    \lesssim h \| \nablas u \|_{\Gamma} \| v\|_{s_h}
    \lesssim h \| f \|_{\Gamma} \| v \|_{s_h}
    \label{eq:estimate-III}
  \end{align}
  Now again considering the case
  $v = \pi_h \phi^e$, we can estimate~\eqref{eq:estimate-III} further
  via
  \begin{align}
    \| \pi_h \phi^e \|_{s_h}
    \lesssim
    h^{1/2} \| \nabla \phi^e \|_{U_{\delta}(\Gamma)}
    +
    \| \pi_h \phi^e - \phi^e \|_{s_h}
    \lesssim
    h \| \nabla \phi \|_{\Gamma} + 
    h^2 \| \phi \|_{2,\Gamma}
  \end{align}
  where we pick a $\delta$ in the stability
  estimate~\eqref{eq:stability-estimate-for-extension}
  such that $\mcT_h \subseteq U_{\delta}(\Gamma)$
  and $\delta \lesssim h$.
  This concludes the proof.
\end{proof}
\begin{remark}
The previous Lemma shows that
the consistency error can be improved by one order of $h$
when the ``consistency error functional''
 $R_h(\cdot) = l_h(\cdot) - A_h(u^e, \cdot)$ is evaluated for
special functions $v \in V_h$ which are
interpolation of smooth functions $v \in H^2(\Gamma)$.
It is precisely this improved estimate which will allow us to prove 
optimal $L^2$ error estimates using a duality argument, despite the 
fact that the consistency error is generally of order $h$.
\end{remark}

\subsection{A Priori Error Estimates}

\begin{theorem} The following a priori error estimates hold
  \label{thm:aprioriest}
  \begin{align}
    \label{eq:energyest}
    \| u^e - u_h \|_{A_h} &\lesssim h \| f \|_{\Gamma}
    \\ \label{eq:ltwoest}
    \| u^e - u_h \|_{\Gammah} &\lesssim h^2 \| f \|_{\Gamma}
  \end{align}
\end{theorem}
\begin{proof} 
  With the elliptic regularity estimate \eqref{eq:ellreg},
  \eqref{eq:energyest} is a direct consequence of 
  the interpolation estimate and
  the estimate~\eqref{eq:consistency-error-est-primal} of the consistency error arising in the Strang
  Lemma~\ref{lem:strang}.
  To prove~\eqref{eq:ltwoest}, we use the standard Aubin-Nitsche
  trick in combination with the improved
  estimate~\eqref{eq:consistency-error-est-dual}.
  More precisly, let $\psi \in L^2(\Gamma)$ and take $\phi \in
  H^2(\Gamma)$ satisfying $-\Delta_{\Gamma} \phi = \psi$
  and the elliptic regularity estimate
  $\| \phi \|_{2,\Gamma} \lesssim \| \psi \|_{\Gamma}$.
  We define $e = u - u_h^l$ and add and subtract suitable terms
  to derive the following error representation
  \begin{align}
    (e, \psi)_{\Gamma} 
    &= a(e, \phi)
    = a(e, \phi) 
    - A_h(e^e, \phi^e)
    + A_h(e^e, \phi^e)
    \\
    &= 
    \bigl(
    a(e, \phi) 
    - A_h(e^e, \phi^e)
    \bigr)
    + A_h(e^e, \phi^e - \pi_h \phi^e)
    + \bigl(
    A_h(u, \pi_h \phi^e)
    - l_h(\pi_h \phi^e)
    \bigr)
    \\
    &= 
    I + II + III
  \end{align}
  Term $III$ is precisely the one appearing in the improved
  consistency error estimate~\eqref{eq:consistency-error-est-dual},
  and consequently
  \begin{align}
   III 
   \lesssim 
   h^2 \|f \|_{\Gamma} \|\phi\|_{2,\Gamma}
   \lesssim 
   h^2 \|f \|_{\Gamma} \|\psi\|_{\Gamma}
  \end{align}
  The second term can be estimated by combining interpolation and energy norm
  estimates:
  \begin{align}
    II 
    \lesssim \| u^e - u_h \|_{A_h} \| \phi^e - \pi_h \phi^e
    \|_{A_h} \
    \lesssim h \| f \|_{\Gamma} h \| \phi \|_{2,\Gamma}
    \lesssim h \| f \|_{\Gamma} h \| \psi \|_{\Gamma}
  \end{align}
  To derive a bound for the remaining term $I$, we first 
  split of the error contributions introduced by the normal
  part of the gradient and the stabilization $s_h(\cdot, \cdot)$:
  \begin{align}
    I 
    &= 
    \bigl(
    (\nablas e, \nablas \phi)_{\Gamma}
    - (\nablash e^e, \nablash \phi^e)_{\Gamma}
    \bigr)
    - ( \Qsh \nablas e^e, \Qsh \nablas \phi^e)_{\Gamma}
    - s_h(e^e, \phi^e)
    \\
    &= I_a + I_b + I_c
  \end{align}
  Now we proceed exactly as in the proof of
  Lemma~\ref{lem:consistency-error-est}. More precisely, following the
  derivation of estimates for Term $II_a$ and $II_b$ in~\eqref{eq:ah-split}, we see
  that
  \begin{align}
    I_a 
    &\lesssim h^2 \| \nablas e \|_{\Gamma} \| \nablas \phi \|_{\Gamma}
    \lesssim h^3 \| f \|_{\Gamma} \| \psi \|_{\Gamma}
    \\
    I_b 
    &\lesssim 
    \| \nablas e^e \|_{\Gammah} \| \Qsh \nabla \phi^e \|_{\Gammah}
    \lesssim
    h \| f \|_{\Gamma} h \| \nablas \phi \|_{\Gamma}
    \lesssim
    h^2 \| f \|_{\Gamma} \| \psi \|_{\Gamma}
  \end{align}
  Similar as before, we have
  \begin{align}
    I_c 
    \lesssim 
    \| e^e \|_{s_h} 
    h^{1/2} \| \nabla \phi^e \|_{\mcT_h}
    \lesssim
    h \| f \|_{\Gamma}
    h \| \nabla \phi \|_{\Gamma}
    \lesssim
    h^2 \| f \|_{\Gamma}
    \| \psi \|_{\Gamma}
  \end{align}
  Now collecting all the estimates for $I$--$III$, dividing by
  $\| \psi \|_{\Gamma}$ and taking the supremum over $\psi \in
  L^2(\Gamma)$ concludes the proof.
\end{proof} 

\section{Numerical Results}
\label{sec:numerical_results}
This section is devoted to a series of numerical experiments
which corroborate the theoretical findings
and assess the effect of the proposed stabilization on
the accuracy of the discrete solution and the conditioning of the
discrete system. 
First, a convergence study for two test cases is conducted, 
where we also examine and compare the effect of the stabilization parameter
on the accuracy of the computed solution.
In the second series of experiments, we investigate the sensitivity of the condition
number with respect to both the surface positioning in the background
mesh and the stabilization parameter $\tau$.
In all studies, we compare the proposed full gradient
stabilization with alternative approaches to cure the discrete system
from being ill-conditioned.

\subsection{Convergence Rate Tests}
\label{ssec:convergence-rate-tests}
Following the numerical examples presented
in~\cite{BurmanHansboLarsonEtAl2015b},
we consider two test cases for the Laplace-Beltrami-type
problem
\begin{align}
  -\Delta_{\Gamma} u + u = f \quad \text{on } \Gamma
  \label{eq:laplace-beltrami-type-problem}
\end{align}
with given analytical reference solution $u$
and surface $\Gamma = \{ x \in \RR^3 : \phi(x) = 0 \}$
defined by 
a known smooth scalar function $\phi$
with $\nabla \phi(x) \neq 0 \,\forall x \in \Gamma$.
The corresponding right-hand side $f$
can be computed using the following representation of the
Laplace-Beltrami operator
\begin{align}
    \Delta_{\Gamma} u = \Delta u - n_{\Gamma} \cdot \nabla \otimes\nabla u \,
    n_{\Gamma} - \textrm{tr}(\nabla n_{\Gamma}) \nabla u \cdot
    n_{\Gamma}
\end{align}
For the first test example (Example 1) we chose
\begin{equation}
  \left\{
    \begin{aligned}
      u_1 &= 
      \sin\left(\dfrac{\pi x }{2}\right)
      \sin\left(\dfrac{\pi y }{2}\right)
      \sin\left(\dfrac{\pi z }{2}\right)
      \\
      \phi_1 &=  x^2 + y^2 + z^2 - 1
    \end{aligned}
  \right.
\end{equation}
while in the second example (Example 2), we consider the problem defined by
\begin{equation}
  \left\{
    \begin{aligned}
      u_2 &=  xy - 5y + z + xz
      \\
      \phi_2 &=  
      (x^2 - 1)^2
      +
      (y^2 - 1)^2
      +
      (z^2 - 1)^2
      +
      (x^2 + y^2 - 4)^2
      +
      (x^2 + z^2 - 4)^2
      \\
      &\quad +
      (y^2 + z^2 - 4)^2
      - 16
    \end{aligned}
  \right.
\end{equation}
The computed solutions for Example 1 and Example 2 are shown in Figure~\ref{fig:solution-examples}.
In the first convergence experiment,
the tangential gradient form
$a^1_h(v,w) = (\nablash v, \nablas w)_{\mcK_h}$ combined with either 
the full gradient stabilization $s_h$ or the face-based stabilization
$j_h$ is used. 
For the second convergence experiment, we consider the full gradient form
$a_h^2(v,w) = (\nabla v, \nabla w)_{\mcK_h}$ instead.

Starting from a structured mesh $\widetilde{\mcT}_0$ for $\Omega =
[-a,a]^3$ with $a$ large enough such that $ \Gamma \subseteq \Omega$,
a sequence of meshes
$\{\mcT_k\}_{k=0}^5$ is generated 
for each test case by successively refining
$\widetilde{\mcT}_0$ and extracting the corresponding active
mesh as defined by~\eqref{eq:narrow-band-mesh}. 
Based on the manufactured exact solutions, the experimental order of
convergence (EOC) is then calculated by
\begin{align*}
    \text{EOC}(k) = \dfrac{\log(E_{k-1}/E_{k})}{\log(2)}
\end{align*}
where $E_k$ denotes the error of the numerical
solution $u_k$ at refinement level $k$ 
measured in either the
$\| \cdot \|_{H^1(\Gamma_h)}$ or $\|\cdot\|_{L^2(\Gamma_h)}$ norm.
To examine the geometric error contributed to the non-vanishing
normal gradient component in the full gradient form $a^2_h$,
we also compute in both convergence studies
the error for the
unstabilized discretization schemes given by $a^1_h$ and $a^2_h$
and $\tau = 0$.

For the two test cases, the computed errors for the sequence of
refined meshes are summarized in
Table~\ref{tab:convergence-rates-example-1-unstabilized}--\ref{tab:convergence-rates-example-1-stabilized}
and 
Table~\ref{tab:convergence-rates-example-2-unstabilized}--\ref{tab:convergence-rates-example-2-stabilized}, respectively.
In all cases, the observed EOC confirms the first-order and second-order
convergences rates as predicted by  
Theorem~\ref{thm:aprioriest} and the
corresponding a priori error estimates 
for the unstabilized
full gradient form derived in ~\cite{DeckelnickElliottRanner2013,Reusken2014} 
and 
for the face-based stabilized tangential form analyzed in~ \cite{BurmanHansboLarson2015}.
\begin{figure}[htb]
    \begin{subfigure}[b]{0.49\textwidth}
    \includegraphics[width=1.0\textwidth]{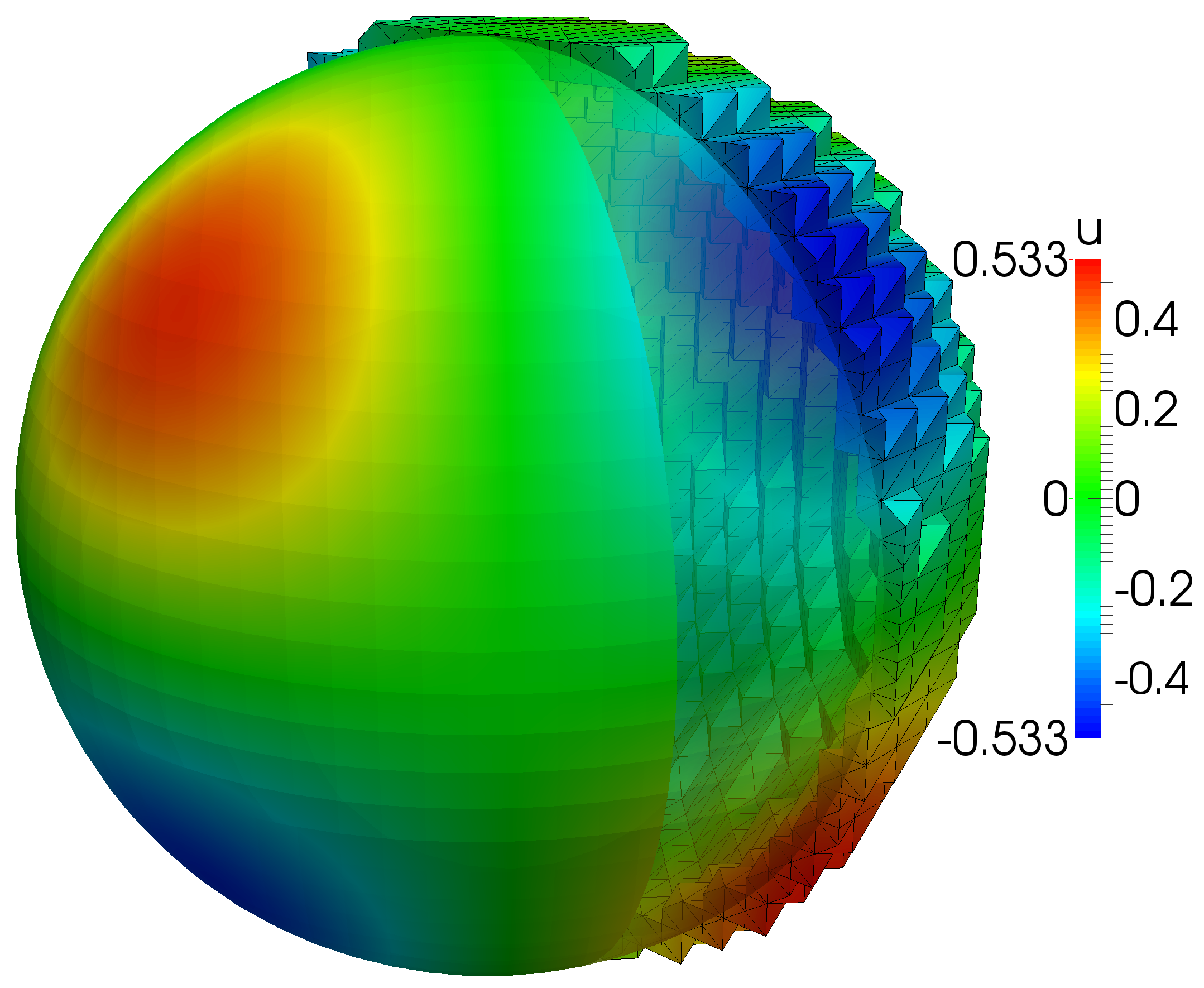}
    \end{subfigure}
    \begin{subfigure}[b]{0.49\textwidth}
    \includegraphics[width=1.0\textwidth]{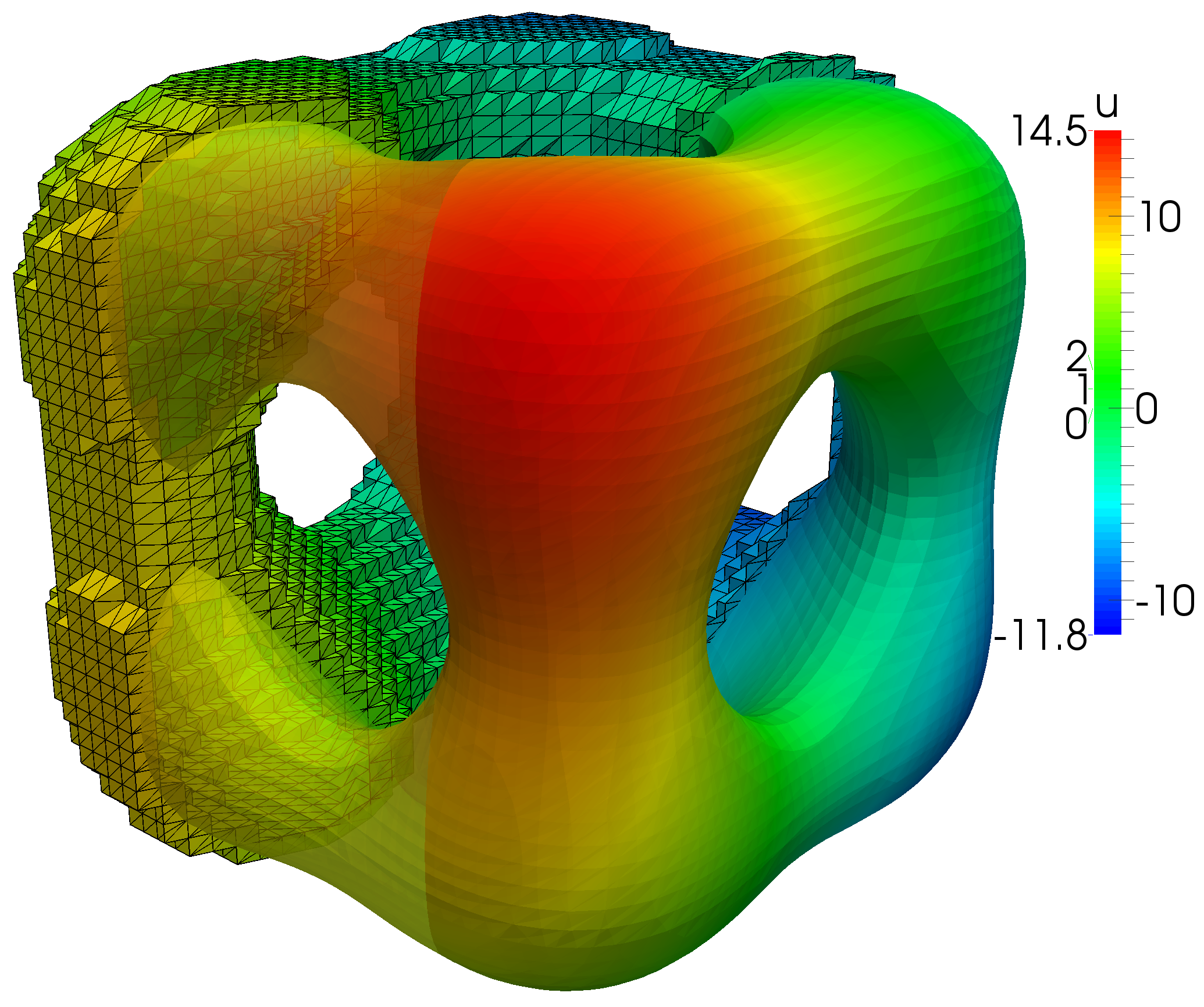}
    \end{subfigure}
    \hspace{0.5ex}
    \caption{Solution plots from the two convergence studies.
      Each plot shows both the approximation $u_h$ as
      computed on the active mesh $\mcT_h$ and the restriction of
      $u_h$ to the surface mesh $\mcK_h$.
      (Left) Solution for Example 1 computed on $\mcT_3$ with $h \approx 7.7\cdot10^{-2}$
      using the face-based stabilized tangential form $a^1 + \tau j_h$ with $\tau = 0.01$.
      (Right) Solution for Example 2 computed on $\mcT_3$ with
      $h \approx 1.15 \cdot 10^{-1}$ using the full gradient stabilized
      full gradient form $a^2 + \tau s_h$ with $\tau = 1.0$.
    }
    \label{fig:solution-examples}
\end{figure}
A closer look at
Table~\ref{tab:convergence-rates-example-1-unstabilized} reveals
that the method based on the unstabilized full gradient form leads, as expected, to a slightly higher error when compared to its tangential
gradient counterpart, in agreement with the numerical results presented
in~\cite{Reusken2014}. A similar increase of the discretization error
can be observed for the second test example, see
Table~\ref{tab:convergence-rates-example-2-unstabilized}.

Turning to the comparison of the full gradient stabilization $s_h$ and 
the face-based stabilization $j_h$ presented in
Table~\ref{tab:convergence-rates-example-1-stabilized}
and Table~\ref{tab:convergence-rates-example-2-stabilized},
we observe that the choice of the stabilization parameter $\tau$
is much less critical for the accuracy
of full gradient stabilized methods then for
the face-based stabilized counterparts, in particular when
the error is measured in the $L^2$ norm. Indeed, while the
$L^2$ error increases only by a factor of  $\sim$$0.1$
when $\tau$ changes from $0.01$ to $1.0$ in the full gradient
stabilization, the error grows by a factor of $\sim$$12$-$14$ when
the face-based stabilization is used.

\begin{table}[htb]
  \footnotesize
  \begin{subtable}[t]{0.49\textwidth}
    \centering
    \begin {tabular}{cr<{\pgfplotstableresetcolortbloverhangright }@{}l<{\pgfplotstableresetcolortbloverhangleft }cr<{\pgfplotstableresetcolortbloverhangright }@{}l<{\pgfplotstableresetcolortbloverhangleft }c}%
\toprule $k$&\multicolumn {2}{c}{$\|u_k - u \|_{1,\Gamma _h}$}&EOC&\multicolumn {2}{c}{$\|u_k - u \|_{\Gamma _h}$}&EOC\\\midrule %
\pgfutilensuremath {0}&$1.36$&$\cdot 10^{0}$&--&$2.59$&$\cdot 10^{-1}$&--\\%
\pgfutilensuremath {1}&$7.72$&$\cdot 10^{-1}$&\pgfutilensuremath {0.82}&$7.39$&$\cdot 10^{-2}$&\pgfutilensuremath {1.81}\\%
\pgfutilensuremath {2}&$3.85$&$\cdot 10^{-1}$&\pgfutilensuremath {1.00}&$1.87$&$\cdot 10^{-2}$&\pgfutilensuremath {1.98}\\%
\pgfutilensuremath {3}&$1.92$&$\cdot 10^{-1}$&\pgfutilensuremath {1.00}&$4.56$&$\cdot 10^{-3}$&\pgfutilensuremath {2.04}\\%
\pgfutilensuremath {4}&$9.59$&$\cdot 10^{-2}$&\pgfutilensuremath {1.00}&$1.13$&$\cdot 10^{-3}$&\pgfutilensuremath {2.01}\\%
\pgfutilensuremath {5}&$4.80$&$\cdot 10^{-2}$&\pgfutilensuremath {1.00}&$2.83$&$\cdot 10^{-4}$&\pgfutilensuremath {2.00}\\\bottomrule %
\end {tabular}%

    \\[1ex]
      \caption{Tangential gradient}
  \end{subtable}
  \begin{subtable}[t]{0.49\textwidth}
    \centering
    \begin {tabular}{cr<{\pgfplotstableresetcolortbloverhangright }@{}l<{\pgfplotstableresetcolortbloverhangleft }cr<{\pgfplotstableresetcolortbloverhangright }@{}l<{\pgfplotstableresetcolortbloverhangleft }c}%
\toprule $k$&\multicolumn {2}{c}{$\|u_k - u \|_{1,\Gamma _h}$}&EOC&\multicolumn {2}{c}{$\|u_k - u \|_{\Gamma _h}$}&EOC\\\midrule %
\pgfutilensuremath {0}&$1.36$&$\cdot 10^{0}$&--&$2.66$&$\cdot 10^{-1}$&--\\%
\pgfutilensuremath {1}&$8.70$&$\cdot 10^{-1}$&\pgfutilensuremath {0.65}&$1.18$&$\cdot 10^{-1}$&\pgfutilensuremath {1.18}\\%
\pgfutilensuremath {2}&$4.55$&$\cdot 10^{-1}$&\pgfutilensuremath {0.94}&$3.39$&$\cdot 10^{-2}$&\pgfutilensuremath {1.79}\\%
\pgfutilensuremath {3}&$2.36$&$\cdot 10^{-1}$&\pgfutilensuremath {0.95}&$8.99$&$\cdot 10^{-3}$&\pgfutilensuremath {1.91}\\%
\pgfutilensuremath {4}&$1.19$&$\cdot 10^{-1}$&\pgfutilensuremath {0.99}&$2.27$&$\cdot 10^{-3}$&\pgfutilensuremath {1.99}\\%
\pgfutilensuremath {5}&$5.93$&$\cdot 10^{-2}$&\pgfutilensuremath {1.00}&$5.68$&$\cdot 10^{-4}$&\pgfutilensuremath {2.00}\\\bottomrule %
\end {tabular}%

    \\[1ex]
      \caption{Full gradient}
  \end{subtable}
  \\
  \caption{Convergence rates for Example 1 comparing the
    unstabilized tangential gradient formulation with the unstabilized
  full gradient method.}
  \label{tab:convergence-rates-example-1-unstabilized}
\end{table}
\begin{table}[htb]
  \footnotesize
  \begin{subtable}[t]{0.49\textwidth}
    \centering
    \begin {tabular}{cr<{\pgfplotstableresetcolortbloverhangright }@{}l<{\pgfplotstableresetcolortbloverhangleft }cr<{\pgfplotstableresetcolortbloverhangright }@{}l<{\pgfplotstableresetcolortbloverhangleft }c}%
\toprule $k$&\multicolumn {2}{c}{$\|u_k - u \|_{1,\Gamma _h}$}&EOC&\multicolumn {2}{c}{$\|u_k - u \|_{\Gamma _h}$}&EOC\\\midrule %
\pgfutilensuremath {0}&$1.36$&$\cdot 10^{0}$&--&$2.61$&$\cdot 10^{-1}$&--\\%
\pgfutilensuremath {1}&$7.69$&$\cdot 10^{-1}$&\pgfutilensuremath {0.82}&$7.54$&$\cdot 10^{-2}$&\pgfutilensuremath {1.79}\\%
\pgfutilensuremath {2}&$3.84$&$\cdot 10^{-1}$&\pgfutilensuremath {1.00}&$1.89$&$\cdot 10^{-2}$&\pgfutilensuremath {2.00}\\%
\pgfutilensuremath {3}&$1.92$&$\cdot 10^{-1}$&\pgfutilensuremath {1.00}&$4.58$&$\cdot 10^{-3}$&\pgfutilensuremath {2.04}\\%
\pgfutilensuremath {4}&$9.59$&$\cdot 10^{-2}$&\pgfutilensuremath {1.00}&$1.13$&$\cdot 10^{-3}$&\pgfutilensuremath {2.02}\\%
\pgfutilensuremath {5}&$4.80$&$\cdot 10^{-2}$&\pgfutilensuremath {1.00}&$2.84$&$\cdot 10^{-4}$&\pgfutilensuremath {2.00}\\\bottomrule %
\end {tabular}%

    \\[1ex]
    \begin {tabular}{cr<{\pgfplotstableresetcolortbloverhangright }@{}l<{\pgfplotstableresetcolortbloverhangleft }cr<{\pgfplotstableresetcolortbloverhangright }@{}l<{\pgfplotstableresetcolortbloverhangleft }c}%
\toprule $k$&\multicolumn {2}{c}{$\|u_k - u \|_{1,\Gamma _h}$}&EOC&\multicolumn {2}{c}{$\|u_k - u \|_{\Gamma _h}$}&EOC\\\midrule %
\pgfutilensuremath {0}&$1.37$&$\cdot 10^{0}$&--&$2.76$&$\cdot 10^{-1}$&--\\%
\pgfutilensuremath {1}&$7.72$&$\cdot 10^{-1}$&\pgfutilensuremath {0.83}&$8.58$&$\cdot 10^{-2}$&\pgfutilensuremath {1.69}\\%
\pgfutilensuremath {2}&$3.83$&$\cdot 10^{-1}$&\pgfutilensuremath {1.01}&$2.01$&$\cdot 10^{-2}$&\pgfutilensuremath {2.09}\\%
\pgfutilensuremath {3}&$1.91$&$\cdot 10^{-1}$&\pgfutilensuremath {1.00}&$4.77$&$\cdot 10^{-3}$&\pgfutilensuremath {2.07}\\%
\pgfutilensuremath {4}&$9.56$&$\cdot 10^{-2}$&\pgfutilensuremath {1.00}&$1.16$&$\cdot 10^{-3}$&\pgfutilensuremath {2.04}\\%
\pgfutilensuremath {5}&$4.79$&$\cdot 10^{-2}$&\pgfutilensuremath {1.00}&$2.87$&$\cdot 10^{-4}$&\pgfutilensuremath {2.01}\\\bottomrule %
\end {tabular}%

    \\[1ex]
    \begin {tabular}{cr<{\pgfplotstableresetcolortbloverhangright }@{}l<{\pgfplotstableresetcolortbloverhangleft }cr<{\pgfplotstableresetcolortbloverhangright }@{}l<{\pgfplotstableresetcolortbloverhangleft }c}%
\toprule $k$&\multicolumn {2}{c}{$\|u_k - u \|_{1,\Gamma _h}$}&EOC&\multicolumn {2}{c}{$\|u_k - u \|_{\Gamma _h}$}&EOC\\\midrule %
\pgfutilensuremath {0}&$1.60$&$\cdot 10^{0}$&--&$3.76$&$\cdot 10^{-1}$&--\\%
\pgfutilensuremath {1}&$8.54$&$\cdot 10^{-1}$&\pgfutilensuremath {0.90}&$1.44$&$\cdot 10^{-1}$&\pgfutilensuremath {1.38}\\%
\pgfutilensuremath {2}&$3.95$&$\cdot 10^{-1}$&\pgfutilensuremath {1.11}&$2.93$&$\cdot 10^{-2}$&\pgfutilensuremath {2.30}\\%
\pgfutilensuremath {3}&$1.94$&$\cdot 10^{-1}$&\pgfutilensuremath {1.02}&$6.22$&$\cdot 10^{-3}$&\pgfutilensuremath {2.24}\\%
\pgfutilensuremath {4}&$9.58$&$\cdot 10^{-2}$&\pgfutilensuremath {1.02}&$1.36$&$\cdot 10^{-3}$&\pgfutilensuremath {2.19}\\%
\pgfutilensuremath {5}&$4.77$&$\cdot 10^{-2}$&\pgfutilensuremath {1.01}&$3.15$&$\cdot 10^{-4}$&\pgfutilensuremath {2.11}\\\bottomrule %
\end {tabular}%

    \\[1ex]
    \caption{With full gradient stabilization $s_h$}
  \end{subtable}
  \begin{subtable}[t]{0.49\textwidth}
    \centering
    \begin {tabular}{cr<{\pgfplotstableresetcolortbloverhangright }@{}l<{\pgfplotstableresetcolortbloverhangleft }cr<{\pgfplotstableresetcolortbloverhangright }@{}l<{\pgfplotstableresetcolortbloverhangleft }c}%
\toprule $k$&\multicolumn {2}{c}{$\|u_k - u \|_{1,\Gamma _h}$}&EOC&\multicolumn {2}{c}{$\|u_k - u \|_{\Gamma _h}$}&EOC\\\midrule %
\pgfutilensuremath {0}&$1.37$&$\cdot 10^{0}$&--&$2.70$&$\cdot 10^{-1}$&--\\%
\pgfutilensuremath {1}&$7.71$&$\cdot 10^{-1}$&\pgfutilensuremath {0.83}&$8.33$&$\cdot 10^{-2}$&\pgfutilensuremath {1.70}\\%
\pgfutilensuremath {2}&$3.85$&$\cdot 10^{-1}$&\pgfutilensuremath {1.00}&$2.16$&$\cdot 10^{-2}$&\pgfutilensuremath {1.95}\\%
\pgfutilensuremath {3}&$1.92$&$\cdot 10^{-1}$&\pgfutilensuremath {1.00}&$5.36$&$\cdot 10^{-3}$&\pgfutilensuremath {2.01}\\%
\pgfutilensuremath {4}&$9.58$&$\cdot 10^{-2}$&\pgfutilensuremath {1.00}&$1.33$&$\cdot 10^{-3}$&\pgfutilensuremath {2.01}\\%
\pgfutilensuremath {5}&$4.80$&$\cdot 10^{-2}$&\pgfutilensuremath {1.00}&$3.35$&$\cdot 10^{-4}$&\pgfutilensuremath {1.99}\\\bottomrule %
\end {tabular}%

    \\[1ex]
    \begin {tabular}{cr<{\pgfplotstableresetcolortbloverhangright }@{}l<{\pgfplotstableresetcolortbloverhangleft }cr<{\pgfplotstableresetcolortbloverhangright }@{}l<{\pgfplotstableresetcolortbloverhangleft }c}%
\toprule $k$&\multicolumn {2}{c}{$\|u_k - u \|_{1,\Gamma _h}$}&EOC&\multicolumn {2}{c}{$\|u_k - u \|_{\Gamma _h}$}&EOC\\\midrule %
\pgfutilensuremath {0}&$1.51$&$\cdot 10^{0}$&--&$3.46$&$\cdot 10^{-1}$&--\\%
\pgfutilensuremath {1}&$8.46$&$\cdot 10^{-1}$&\pgfutilensuremath {0.84}&$1.49$&$\cdot 10^{-1}$&\pgfutilensuremath {1.22}\\%
\pgfutilensuremath {2}&$4.08$&$\cdot 10^{-1}$&\pgfutilensuremath {1.05}&$4.37$&$\cdot 10^{-2}$&\pgfutilensuremath {1.76}\\%
\pgfutilensuremath {3}&$2.01$&$\cdot 10^{-1}$&\pgfutilensuremath {1.02}&$1.15$&$\cdot 10^{-2}$&\pgfutilensuremath {1.93}\\%
\pgfutilensuremath {4}&$9.96$&$\cdot 10^{-2}$&\pgfutilensuremath {1.01}&$2.91$&$\cdot 10^{-3}$&\pgfutilensuremath {1.98}\\%
\pgfutilensuremath {5}&$4.98$&$\cdot 10^{-2}$&\pgfutilensuremath {1.00}&$7.32$&$\cdot 10^{-4}$&\pgfutilensuremath {1.99}\\\bottomrule %
\end {tabular}%

    \\[1ex]
    \begin {tabular}{cr<{\pgfplotstableresetcolortbloverhangright }@{}l<{\pgfplotstableresetcolortbloverhangleft }cr<{\pgfplotstableresetcolortbloverhangright }@{}l<{\pgfplotstableresetcolortbloverhangleft }c}%
\toprule $k$&\multicolumn {2}{c}{$\|u_k - u \|_{1,\Gamma _h}$}&EOC&\multicolumn {2}{c}{$\|u_k - u \|_{\Gamma _h}$}&EOC\\\midrule %
\pgfutilensuremath {0}&$2.14$&$\cdot 10^{0}$&--&$5.35$&$\cdot 10^{-1}$&--\\%
\pgfutilensuremath {1}&$1.71$&$\cdot 10^{0}$&\pgfutilensuremath {0.32}&$4.55$&$\cdot 10^{-1}$&\pgfutilensuremath {0.24}\\%
\pgfutilensuremath {2}&$8.14$&$\cdot 10^{-1}$&\pgfutilensuremath {1.07}&$2.07$&$\cdot 10^{-1}$&\pgfutilensuremath {1.13}\\%
\pgfutilensuremath {3}&$3.07$&$\cdot 10^{-1}$&\pgfutilensuremath {1.40}&$6.59$&$\cdot 10^{-2}$&\pgfutilensuremath {1.65}\\%
\pgfutilensuremath {4}&$1.20$&$\cdot 10^{-1}$&\pgfutilensuremath {1.35}&$1.78$&$\cdot 10^{-2}$&\pgfutilensuremath {1.88}\\%
\pgfutilensuremath {5}&$5.39$&$\cdot 10^{-2}$&\pgfutilensuremath {1.16}&$4.56$&$\cdot 10^{-3}$&\pgfutilensuremath {1.97}\\\bottomrule %
\end {tabular}%

    \\[1ex]
    \caption{With face-based stabilization $j_h$}
  \end{subtable}
  \caption{Convergence rates for Example 1. The solution is computed from
     a combination of the tangential gradient form $a_h^1(v,w) = (\nablash v, \nablash w)_{\mcK_h}$
    with different stabilizations and penalty parameters. Penalty
    parameter $\tau$ was set to $\tau = 0.01$ (top), $\tau = 0.1$
    (middle), and $\tau = 1.0$ (bottom). }
  \label{tab:convergence-rates-example-1-stabilized}
\end{table}

\begin{table}[htb]
  \footnotesize
  \begin{subtable}[t]{0.49\textwidth}
    \centering
    \begin {tabular}{cr<{\pgfplotstableresetcolortbloverhangright }@{}l<{\pgfplotstableresetcolortbloverhangleft }cr<{\pgfplotstableresetcolortbloverhangright }@{}l<{\pgfplotstableresetcolortbloverhangleft }c}%
\toprule $k$&\multicolumn {2}{c}{$\|u_k - u \|_{1,\Gamma _h}$}&EOC&\multicolumn {2}{c}{$\|u_k - u \|_{\Gamma _h}$}&EOC\\\midrule %
\pgfutilensuremath {0}&$2.24$&$\cdot 10^{1}$&--&$1.49$&$\cdot 10^{1}$&--\\%
\pgfutilensuremath {1}&$7.92$&$\cdot 10^{0}$&\pgfutilensuremath {1.50}&$1.86$&$\cdot 10^{0}$&\pgfutilensuremath {3.00}\\%
\pgfutilensuremath {2}&$3.48$&$\cdot 10^{0}$&\pgfutilensuremath {1.19}&$5.48$&$\cdot 10^{-1}$&\pgfutilensuremath {1.76}\\%
\pgfutilensuremath {3}&$1.72$&$\cdot 10^{0}$&\pgfutilensuremath {1.02}&$1.18$&$\cdot 10^{-1}$&\pgfutilensuremath {2.21}\\%
\pgfutilensuremath {4}&$8.50$&$\cdot 10^{-1}$&\pgfutilensuremath {1.01}&$2.92$&$\cdot 10^{-2}$&\pgfutilensuremath {2.02}\\%
\pgfutilensuremath {5}&$4.24$&$\cdot 10^{-1}$&\pgfutilensuremath {1.00}&$7.32$&$\cdot 10^{-3}$&\pgfutilensuremath {1.99}\\\bottomrule %
\end {tabular}%

    \\[1ex]
      \caption{Tangential gradient}
  \end{subtable}
  \begin{subtable}[t]{0.49\textwidth}
    \centering
    \begin {tabular}{cr<{\pgfplotstableresetcolortbloverhangright }@{}l<{\pgfplotstableresetcolortbloverhangleft }cr<{\pgfplotstableresetcolortbloverhangright }@{}l<{\pgfplotstableresetcolortbloverhangleft }c}%
\toprule $k$&\multicolumn {2}{c}{$\|u_k - u \|_{1,\Gamma _h}$}&EOC&\multicolumn {2}{c}{$\|u_k - u \|_{\Gamma _h}$}&EOC\\\midrule %
\pgfutilensuremath {0}&$2.48$&$\cdot 10^{1}$&--&$1.62$&$\cdot 10^{1}$&--\\%
\pgfutilensuremath {1}&$9.30$&$\cdot 10^{0}$&\pgfutilensuremath {1.41}&$2.28$&$\cdot 10^{0}$&\pgfutilensuremath {2.83}\\%
\pgfutilensuremath {2}&$5.48$&$\cdot 10^{0}$&\pgfutilensuremath {0.76}&$5.41$&$\cdot 10^{-1}$&\pgfutilensuremath {2.07}\\%
\pgfutilensuremath {3}&$2.84$&$\cdot 10^{0}$&\pgfutilensuremath {0.95}&$1.49$&$\cdot 10^{-1}$&\pgfutilensuremath {1.86}\\%
\pgfutilensuremath {4}&$1.46$&$\cdot 10^{0}$&\pgfutilensuremath {0.96}&$4.06$&$\cdot 10^{-2}$&\pgfutilensuremath {1.88}\\%
\pgfutilensuremath {5}&$7.38$&$\cdot 10^{-1}$&\pgfutilensuremath {0.98}&$1.04$&$\cdot 10^{-2}$&\pgfutilensuremath {1.96}\\\bottomrule %
\end {tabular}%

    \\[1ex]
      \caption{Full gradient}
  \end{subtable}
  \\
  \caption{Convergence rates for Example 2 comparing the
    unstabilized tangential gradient formulation with the unstabilized
  full gradient method.}
  \label{tab:convergence-rates-example-2-unstabilized}
\end{table}
\begin{table}[htb]
  \footnotesize
  \begin{subtable}[t]{0.49\textwidth}
    \centering
    \begin {tabular}{cr<{\pgfplotstableresetcolortbloverhangright }@{}l<{\pgfplotstableresetcolortbloverhangleft }cr<{\pgfplotstableresetcolortbloverhangright }@{}l<{\pgfplotstableresetcolortbloverhangleft }c}%
\toprule $k$&\multicolumn {2}{c}{$\|u_k - u \|_{1,\Gamma _h}$}&EOC&\multicolumn {2}{c}{$\|u_k - u \|_{\Gamma _h}$}&EOC\\\midrule %
\pgfutilensuremath {0}&$2.49$&$\cdot 10^{1}$&--&$1.63$&$\cdot 10^{1}$&--\\%
\pgfutilensuremath {1}&$9.29$&$\cdot 10^{0}$&\pgfutilensuremath {1.42}&$2.31$&$\cdot 10^{0}$&\pgfutilensuremath {2.82}\\%
\pgfutilensuremath {2}&$5.48$&$\cdot 10^{0}$&\pgfutilensuremath {0.76}&$5.46$&$\cdot 10^{-1}$&\pgfutilensuremath {2.08}\\%
\pgfutilensuremath {3}&$2.84$&$\cdot 10^{0}$&\pgfutilensuremath {0.95}&$1.50$&$\cdot 10^{-1}$&\pgfutilensuremath {1.86}\\%
\pgfutilensuremath {4}&$1.46$&$\cdot 10^{0}$&\pgfutilensuremath {0.96}&$4.10$&$\cdot 10^{-2}$&\pgfutilensuremath {1.88}\\%
\pgfutilensuremath {5}&$7.38$&$\cdot 10^{-1}$&\pgfutilensuremath {0.98}&$1.05$&$\cdot 10^{-2}$&\pgfutilensuremath {1.96}\\\bottomrule %
\end {tabular}%

    \\[1ex]
    \begin {tabular}{cr<{\pgfplotstableresetcolortbloverhangright }@{}l<{\pgfplotstableresetcolortbloverhangleft }cr<{\pgfplotstableresetcolortbloverhangright }@{}l<{\pgfplotstableresetcolortbloverhangleft }c}%
\toprule $k$&\multicolumn {2}{c}{$\|u_k - u \|_{1,\Gamma _h}$}&EOC&\multicolumn {2}{c}{$\|u_k - u \|_{\Gamma _h}$}&EOC\\\midrule %
\pgfutilensuremath {0}&$2.61$&$\cdot 10^{1}$&--&$1.76$&$\cdot 10^{1}$&--\\%
\pgfutilensuremath {1}&$9.28$&$\cdot 10^{0}$&\pgfutilensuremath {1.49}&$2.70$&$\cdot 10^{0}$&\pgfutilensuremath {2.71}\\%
\pgfutilensuremath {2}&$5.48$&$\cdot 10^{0}$&\pgfutilensuremath {0.76}&$6.18$&$\cdot 10^{-1}$&\pgfutilensuremath {2.13}\\%
\pgfutilensuremath {3}&$2.84$&$\cdot 10^{0}$&\pgfutilensuremath {0.95}&$1.69$&$\cdot 10^{-1}$&\pgfutilensuremath {1.87}\\%
\pgfutilensuremath {4}&$1.46$&$\cdot 10^{0}$&\pgfutilensuremath {0.96}&$4.59$&$\cdot 10^{-2}$&\pgfutilensuremath {1.89}\\%
\pgfutilensuremath {5}&$7.38$&$\cdot 10^{-1}$&\pgfutilensuremath {0.98}&$1.17$&$\cdot 10^{-2}$&\pgfutilensuremath {1.97}\\\bottomrule %
\end {tabular}%

    \\[1ex]
    \begin {tabular}{cr<{\pgfplotstableresetcolortbloverhangright }@{}l<{\pgfplotstableresetcolortbloverhangleft }cr<{\pgfplotstableresetcolortbloverhangright }@{}l<{\pgfplotstableresetcolortbloverhangleft }c}%
\toprule $k$&\multicolumn {2}{c}{$\|u_k - u \|_{1,\Gamma _h}$}&EOC&\multicolumn {2}{c}{$\|u_k - u \|_{\Gamma _h}$}&EOC\\\midrule %
\pgfutilensuremath {0}&$3.55$&$\cdot 10^{1}$&--&$2.66$&$\cdot 10^{1}$&--\\%
\pgfutilensuremath {1}&$1.34$&$\cdot 10^{1}$&\pgfutilensuremath {1.41}&$8.44$&$\cdot 10^{0}$&\pgfutilensuremath {1.66}\\%
\pgfutilensuremath {2}&$5.99$&$\cdot 10^{0}$&\pgfutilensuremath {1.16}&$2.09$&$\cdot 10^{0}$&\pgfutilensuremath {2.01}\\%
\pgfutilensuremath {3}&$2.93$&$\cdot 10^{0}$&\pgfutilensuremath {1.03}&$5.54$&$\cdot 10^{-1}$&\pgfutilensuremath {1.92}\\%
\pgfutilensuremath {4}&$1.47$&$\cdot 10^{0}$&\pgfutilensuremath {0.99}&$1.43$&$\cdot 10^{-1}$&\pgfutilensuremath {1.96}\\%
\pgfutilensuremath {5}&$7.39$&$\cdot 10^{-1}$&\pgfutilensuremath {0.99}&$3.60$&$\cdot 10^{-2}$&\pgfutilensuremath {1.99}\\\bottomrule %
\end {tabular}%

    \\[1ex]
    \caption{With full gradient stabilization}
  \end{subtable}
  \begin{subtable}[t]{0.49\textwidth}
    \centering
    \begin {tabular}{cr<{\pgfplotstableresetcolortbloverhangright }@{}l<{\pgfplotstableresetcolortbloverhangleft }cr<{\pgfplotstableresetcolortbloverhangright }@{}l<{\pgfplotstableresetcolortbloverhangleft }c}%
\toprule $k$&\multicolumn {2}{c}{$\|u_k - u \|_{1,\Gamma _h}$}&EOC&\multicolumn {2}{c}{$\|u_k - u \|_{\Gamma _h}$}&EOC\\\midrule %
\pgfutilensuremath {0}&$2.47$&$\cdot 10^{1}$&--&$1.63$&$\cdot 10^{1}$&--\\%
\pgfutilensuremath {1}&$9.10$&$\cdot 10^{0}$&\pgfutilensuremath {1.44}&$2.35$&$\cdot 10^{0}$&\pgfutilensuremath {2.79}\\%
\pgfutilensuremath {2}&$5.41$&$\cdot 10^{0}$&\pgfutilensuremath {0.75}&$5.71$&$\cdot 10^{-1}$&\pgfutilensuremath {2.04}\\%
\pgfutilensuremath {3}&$2.82$&$\cdot 10^{0}$&\pgfutilensuremath {0.94}&$1.63$&$\cdot 10^{-1}$&\pgfutilensuremath {1.81}\\%
\pgfutilensuremath {4}&$1.45$&$\cdot 10^{0}$&\pgfutilensuremath {0.96}&$4.51$&$\cdot 10^{-2}$&\pgfutilensuremath {1.86}\\%
\pgfutilensuremath {5}&$7.34$&$\cdot 10^{-1}$&\pgfutilensuremath {0.98}&$1.16$&$\cdot 10^{-2}$&\pgfutilensuremath {1.95}\\\bottomrule %
\end {tabular}%

    \\[1ex]
    \begin {tabular}{cr<{\pgfplotstableresetcolortbloverhangright }@{}l<{\pgfplotstableresetcolortbloverhangleft }cr<{\pgfplotstableresetcolortbloverhangright }@{}l<{\pgfplotstableresetcolortbloverhangleft }c}%
\toprule $k$&\multicolumn {2}{c}{$\|u_k - u \|_{1,\Gamma _h}$}&EOC&\multicolumn {2}{c}{$\|u_k - u \|_{\Gamma _h}$}&EOC\\\midrule %
\pgfutilensuremath {0}&$2.43$&$\cdot 10^{1}$&--&$1.68$&$\cdot 10^{1}$&--\\%
\pgfutilensuremath {1}&$8.58$&$\cdot 10^{0}$&\pgfutilensuremath {1.50}&$2.94$&$\cdot 10^{0}$&\pgfutilensuremath {2.52}\\%
\pgfutilensuremath {2}&$5.16$&$\cdot 10^{0}$&\pgfutilensuremath {0.73}&$8.75$&$\cdot 10^{-1}$&\pgfutilensuremath {1.75}\\%
\pgfutilensuremath {3}&$2.75$&$\cdot 10^{0}$&\pgfutilensuremath {0.91}&$2.97$&$\cdot 10^{-1}$&\pgfutilensuremath {1.56}\\%
\pgfutilensuremath {4}&$1.42$&$\cdot 10^{0}$&\pgfutilensuremath {0.95}&$8.82$&$\cdot 10^{-2}$&\pgfutilensuremath {1.75}\\%
\pgfutilensuremath {5}&$7.23$&$\cdot 10^{-1}$&\pgfutilensuremath {0.98}&$2.36$&$\cdot 10^{-2}$&\pgfutilensuremath {1.90}\\\bottomrule %
\end {tabular}%

    \\[1ex]
    \begin {tabular}{cr<{\pgfplotstableresetcolortbloverhangright }@{}l<{\pgfplotstableresetcolortbloverhangleft }cr<{\pgfplotstableresetcolortbloverhangright }@{}l<{\pgfplotstableresetcolortbloverhangleft }c}%
\toprule $k$&\multicolumn {2}{c}{$\|u_k - u \|_{1,\Gamma _h}$}&EOC&\multicolumn {2}{c}{$\|u_k - u \|_{\Gamma _h}$}&EOC\\\midrule %
\pgfutilensuremath {0}&$2.49$&$\cdot 10^{1}$&--&$1.88$&$\cdot 10^{1}$&--\\%
\pgfutilensuremath {1}&$9.61$&$\cdot 10^{0}$&\pgfutilensuremath {1.37}&$5.35$&$\cdot 10^{0}$&\pgfutilensuremath {1.81}\\%
\pgfutilensuremath {2}&$6.38$&$\cdot 10^{0}$&\pgfutilensuremath {0.59}&$2.86$&$\cdot 10^{0}$&\pgfutilensuremath {0.90}\\%
\pgfutilensuremath {3}&$3.64$&$\cdot 10^{0}$&\pgfutilensuremath {0.81}&$1.29$&$\cdot 10^{0}$&\pgfutilensuremath {1.15}\\%
\pgfutilensuremath {4}&$1.76$&$\cdot 10^{0}$&\pgfutilensuremath {1.04}&$4.52$&$\cdot 10^{-1}$&\pgfutilensuremath {1.51}\\%
\pgfutilensuremath {5}&$8.13$&$\cdot 10^{-1}$&\pgfutilensuremath {1.12}&$1.37$&$\cdot 10^{-1}$&\pgfutilensuremath {1.73}\\\bottomrule %
\end {tabular}%

    \\[1ex]
    \caption{With face-based stabilization}
  \end{subtable}
  \caption{Convergence rates for Example 2. Solution is computed from
    a combination of the full gradient form $a_h^2(v,w) = (\nablas v, \nablas w)_{\mcK_h}$
    with different stabilizations and penalty parameters. Penalty
    parameter $\tau$ was set to $\tau = 0.01$ (top), $\tau = 0.1$
    (middle), and $\tau = 1.0$ (bottom). }
  \label{tab:convergence-rates-example-2-stabilized}
\end{table}

\subsection{Condition Number Tests}
\label{ssec:condition-number-tests}
In our final numerical study,
the dependency of the condition number on 
the mesh size and on the positioning of the surface 
in the background mesh is examined.
Additionally, we compare the proposed full gradient stabilization to
both face-based stabilized schemes and diagonal preconditioning as
alternative approaches to obtain robust and moderate condition
numbers for the discrete systems.

We start our numerical experiment by defining a sequence
$\{\mcT_k\}_{k=1}^6$ of tessellations of $\Omega = [-1.6, 1.6]^3$ with
mesh size $h = 3.2/5 \cdot 2^{-k/2}$.
For each $k$, we generate a family of surfaces
$\{\Gamma_{\delta}\}_{0\leqslant\delta\leqslant 1}$  
by translating the unit-sphere $S^2 = \{ x \in \RR^3 : \| x \| = 1 \}$
along the diagonal $(h,h,h)$; that is,
$\Gamma_{\delta} = S^2 + \delta (h,h,h)$ with $\delta \in [0,1]$.
For $\delta = l/500$, $l=0,\ldots,500$,
we compute the condition number $\kappa_{\delta}(\mcA)$ 
as the ratio
of the absolute value of the largest (in modulus) and smallest (in
modulus), non-zero eigenvalue. 
For the full gradient stabilized full gradient method 
with $\tau = 1.0$,
the minimum, maximum, and the arithmetic mean
of the resulting scaled condition numbers $h^{2}\kappa_{\delta}(A)$ 
for each mesh size $h$ are shown in
Table~\ref{tab:scaled-condition-number}. 

The computed figures in Table~\ref{tab:scaled-condition-number} clearly
confirm the theoretically proven $O(h^{-2})$ bound, independent
of the location of the surface in the background mesh. Additionally,
Figure~\ref{fig:condition_number-example-1}  confirms for
$\mcT_2$ the robustness of the condition number with respect
to the translation parameter $\delta$.
In contrast, the condition number is highly sensitive and clearly
unbounded as a function of $\delta$ if we set the penalty
parameter $\tau$ in~\eqref{eq:Ah-def} to $0$ as the corresponding plot in
Figure~\ref{fig:condition_number-example-1} shows.
The same figure also demonstrates that the discrete system
can be made robust by either diagonally scaling or augmenting the discrete
variational form with the face-based stabilization $j_h$, see
\cite{OlshanskiiReusken2010,Reusken2014}
and \cite{BurmanHansboLarson2015}  for the details.
\begin{table}[htb]
  \centering
  \begin{center}
    \begin {tabular}{r<{\pgfplotstableresetcolortbloverhangright }@{}l<{\pgfplotstableresetcolortbloverhangleft }ccc}%
\toprule \multicolumn {2}{c}{$h$}&$\min _{\delta }\{h^2\kappa _{\delta }(\mathcal {A})\}$&$\max _{\delta }\{h^2\kappa _{\delta }(\mathcal {A})\}$&$\mathrm {mean}_{\delta }\{h^2\kappa _{\delta }(\mathcal {A})\}$\\\midrule %
$1.00$&$\cdot 10^{-1}$&\pgfutilensuremath {1.22}&\pgfutilensuremath {1.61}&\pgfutilensuremath {1.36}\\%
$6.67$&$\cdot 10^{-2}$&\pgfutilensuremath {1.19}&\pgfutilensuremath {1.37}&\pgfutilensuremath {1.27}\\%
$5.00$&$\cdot 10^{-2}$&\pgfutilensuremath {1.21}&\pgfutilensuremath {1.36}&\pgfutilensuremath {1.26}\\%
$3.33$&$\cdot 10^{-2}$&\pgfutilensuremath {1.20}&\pgfutilensuremath {1.26}&\pgfutilensuremath {1.22}\\%
$2.50$&$\cdot 10^{-2}$&\pgfutilensuremath {1.20}&\pgfutilensuremath {1.23}&\pgfutilensuremath {1.21}\\%
$1.67$&$\cdot 10^{-2}$&\pgfutilensuremath {1.21}&\pgfutilensuremath {1.24}&\pgfutilensuremath {1.22}\\\bottomrule %
\end {tabular}%

  \end{center}
  \caption{Minimum, maximum, and arithmetic mean of the scaled condition
  number for various mesh sizes $h$. The full gradient stabilized full gradient method is used with $\tau = 1.0$.}
  \label{tab:scaled-condition-number}
\end{table}

In a final numerical experiment, we assess and compare the effect of the
stability parameter choice $\tau$ for the full gradient and face-based
stabilization on the size and position dependency of the condition number.
In our experiment,  we consider both the tangential gradient form 
$a^1_h(v,w) = (\nablash, v, \nablash w)_{\mcK_h}$ 
and the full gradient form 
$a^2_h(v,w) = (\nabla, v, \nabla w)_{\mcK_h}$ 
augmented with either the full gradient stabilization 
$s_h(v,w) = h^2(\nabla v, \nabla w)_{\mcT_h}$
or the face-based stabilization
$j_h(v,w) = (n_F \cdot \jump{\nabla v}, n_F \cdot \jump{\nabla
w})_{\mcF_h}$.  The condition numbers are computed for
the discretizations defined on $\mcT_2$ and displayed as a
function of $\delta$ in Figure~\ref{fig:condition_number-example-2}.

Varying the stabilization parameter $\tau$ from $10^{-4}$ to $10$,
we clearly observe that the condition number attains a minimum
around $\tau \sim 0.1$ when the face-based stabilization $j_h$ is employed.
Recalling from the convergence experiments that large choices of $\tau$
reduces the accuracy of the face-based stabilized surface method
considerably,
a good choice of $\tau$ should balance
both the accuracy of the numerical scheme
and the size and fluctuation of the condition number.
On the contrary, a satisfactory choice of $\tau$ is much less delicate
for the full gradient stabilization.
Indeed, while $\kappa_{\delta}(\mcA)$ as a function of $\delta$
fluctuates slightly more than for the face-based stabilization, the
condition number reveals itself as a monotonically decreasing function
of $\tau$. 

Finally, we note that the appearance of the normal gradient component
in the full gradient form sometimes has a certain stabilizing effect
on the condition number.
Comparing the magnitude 
of the condition number for a tangential gradient based discrete
system with its full gradient counterpart shows that the condition
number is significantly lower for certain surface positions.
Nevertheless, either diagonally preconditioning or additional
stabilization terms are necessary to obtain fully robust condition
numbers.
\begin{figure}[htb]
  \begin{center}
    \includegraphics[page=1,width=0.49\textwidth]{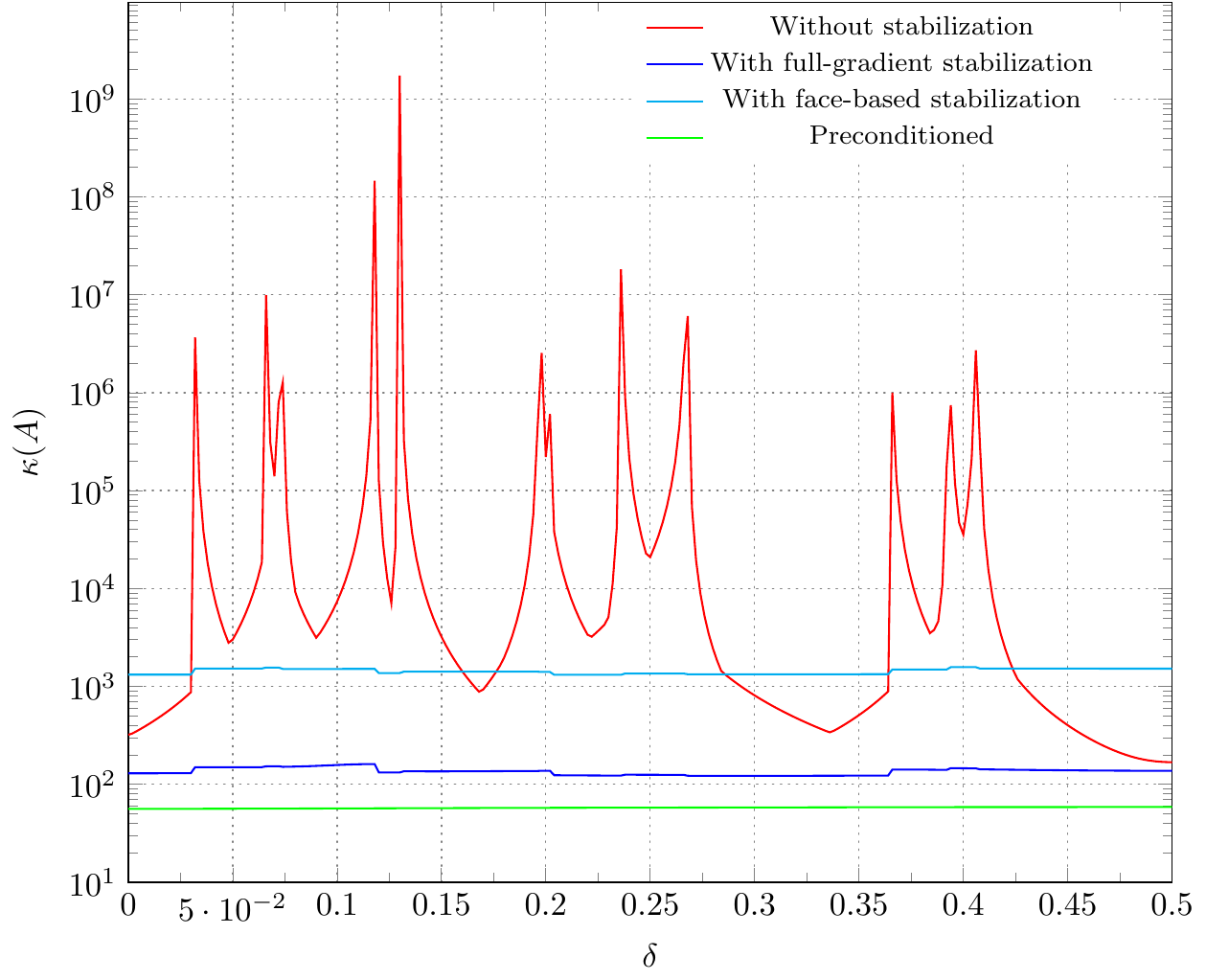}
    \caption{Condition numbers plotted as a function of the position
      parameter $\delta$. 
    \label{fig:condition_number-example-1}
  }
  \end{center}
\end{figure}
\begin{figure}[htb]
  \vspace{2ex}
  \begin{center}
    \begin{subfigure}[t]{0.49\textwidth}
    \includegraphics[page=5,width=1.00\textwidth]{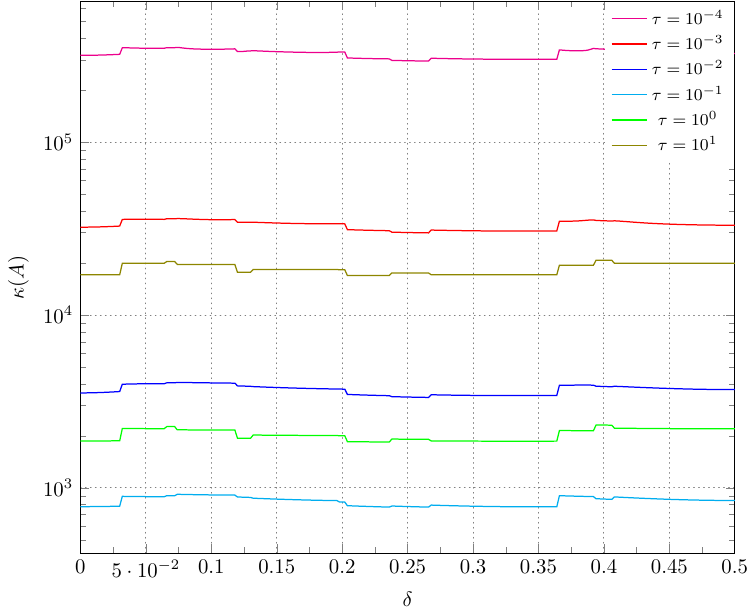}
    \caption{\footnotesize Bilinear form $a^1_h(v,w) + \tau j_h(v,w)$}
    \end{subfigure}
    \begin{subfigure}[t]{0.49\textwidth}
    \includegraphics[page=4,width=1.00\textwidth]{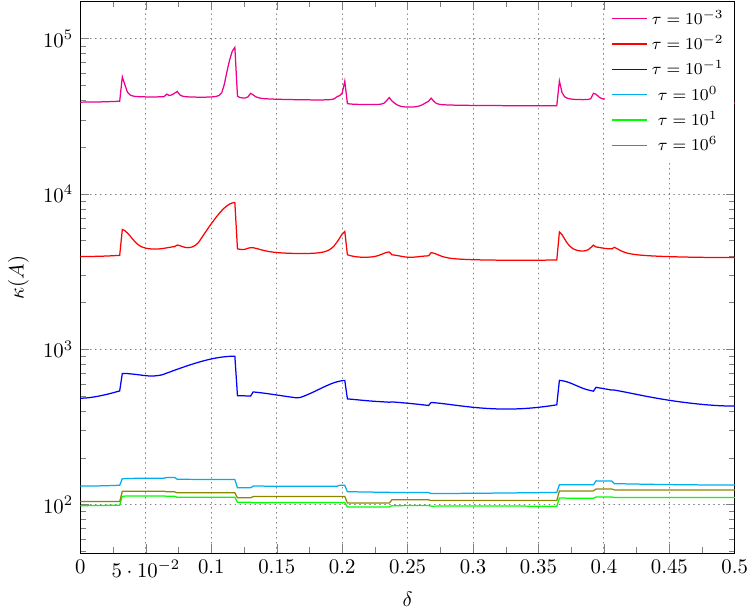}
    \caption{\footnotesize Bilinear form $a^1_h(v,w) + \tau s_h(v,w)$}
    \end{subfigure}
    \\[2ex]
    \begin{subfigure}[t]{0.49\textwidth}
    \includegraphics[page=3,width=1.00\textwidth]{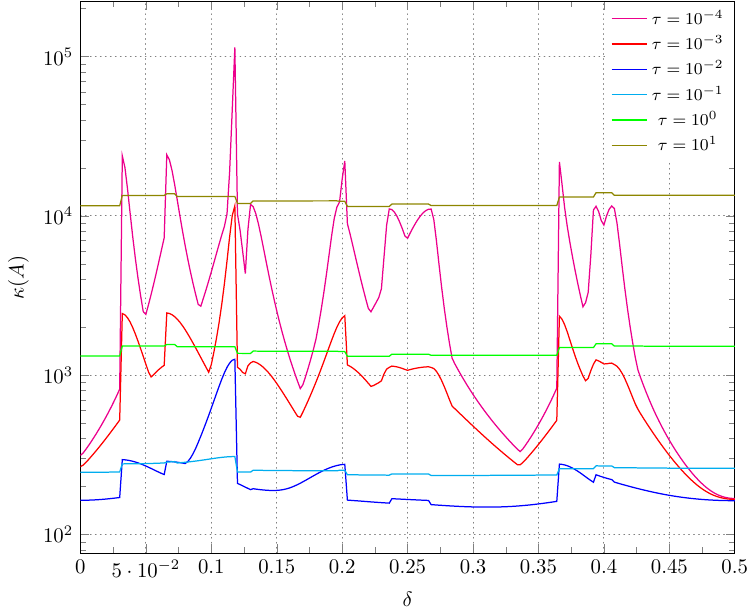}
    \caption{\footnotesize Bilinear form $a^2_h(v,w) + \tau j_h(v,w)$}
    \end{subfigure}
    \begin{subfigure}[t]{0.49\textwidth}
    \includegraphics[page=2,width=1.00\textwidth]{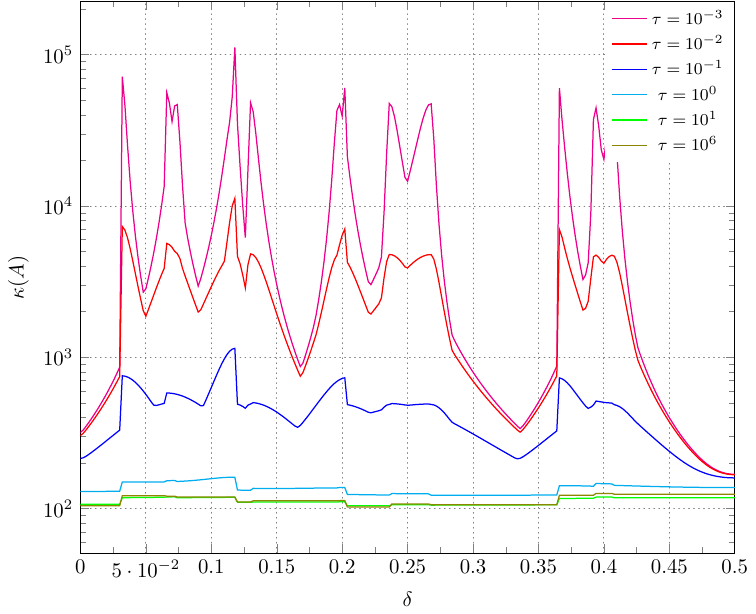}
    \caption{\footnotesize Bilinear form $a^2_h(v,w) + \tau s_h(v,w)$}
    \end{subfigure}
    \caption{Condition numbers plotted as a function of the position parameter $\delta$
      for different combinations of forms,
      stabilizations, and penalty parameters. 
  }
    \label{fig:condition_number-example-2}
  \end{center}
\end{figure}

\section*{Acknowledgements}
This research was supported in part by EPSRC, UK, Grant
No. EP/J002313/1, the Swedish Foundation for Strategic Research Grant
No.\ AM13-0029, the Swedish Research Council Grants Nos.\ 2011-4992,
2013-4708, 2014-4804, and Swedish strategic research programme eSSENCE.

\bibliographystyle{plainnat}
\bibliography{bibliography}

\end{document}